\documentclass[11pt]{amsart}

\usepackage[utf8]{inputenc}
\usepackage[T1]{fontenc}
\usepackage[english]{babel}

\usepackage[foot]{amsaddr}

\usepackage{amsthm}
\usepackage{amsmath}
\usepackage{amssymb}
\usepackage{mathrsfs}
\usepackage{dsfont}
\usepackage{tabularx,multirow}
\usepackage{subcaption}
\usepackage{booktabs}
\usepackage{algorithm}
\usepackage{matlab-prettifier}

\usepackage{graphicx}
\usepackage{float}
\usepackage{comment}
\usepackage{xcolor}

\usepackage{color}
\definecolor{marin}{rgb}   {0.,   0.1,   0.5} 
\definecolor{rouge}{rgb}   {0.8,   0.,   0.} 
\definecolor{sepia}{rgb}   {0.4,   0.25,   0.} 
\definecolor{mag}{rgb}   {0.3,   0,   0.3} 

\usepackage[colorlinks,citecolor=sepia,linkcolor=marin,urlcolor=mag ,bookmarksopen,bookmarksnumbered]{hyperref}

\newcommand{\ic}{\mathrm{i}}
\newcommand{\R}{\mathbb{R}}

\newcommand{\ynull}{\emph{Y0}\;}
\newcommand{\yzwei}{\emph{Y2}\;}
\newcommand{\ydrei}{\emph{Y3}\;}
\newcommand{\yvier}{\emph{Y4}\;}
\newcommand{\yk}{\emph{Yk}\;}

\newcommand{\norm}[1]{\left\lVert#1\right\rVert}

\newtheorem{theorem}{Theorem}[section]
\newtheorem{corollary}[theorem]{Corollary}
\newtheorem{lemma}[theorem]{Lemma}
\newtheorem{proposition}[theorem]{Proposition}
\newtheorem{definition}[theorem]{Definition}
\newtheorem{remark}[theorem]{Remark}
\newtheorem{example}[theorem]{Example}

\numberwithin{equation}{section}

\newcommand{\QED}{\mbox{}\hfill \raisebox{-0.2pt}{\rule{5.6pt}{6pt}\rule{0pt}{0pt}} \medskip\par}

\newenvironment{Proof}{\noindent
   \parindent=0pt\abovedisplayskip = 0.5\abovedisplayskip
    \belowdisplayskip=\abovedisplayskip{\bfseries Proof. }}{\QED\hspace*{.5cm}}
 \newenvironment{ProofOf}[1]{\noindent
    \parindent=0pt\abovedisplayskip = 0.5\abovedisplayskip
    \belowdisplayskip=\abovedisplayskip{\bfseries Proof of  #1. }}{\QED\hspace*{.5cm}}

\newcommand{\myauthors}{J. Bernier, R. Häberli, and G. Vilmart}

\usepackage{fancyhdr}
\renewcommand{\sectionmark}[1]{%
\markboth{\MakeUppercase{#1}}{\MakeUppercase{#1}}%
}

\begin{document}

\title[Splitting methods for linear Schrödinger equations with compatibility conditions]{Arbitrary high order splitting methods for linear {S}chrödinger equations with non-trivial compatibility conditions}

\author[J. Bernier]{Joackim Bernier$^1$}
\address{$^1$Nantes Universit\'e, CNRS, Laboratoire de Math\'ematiques Jean Leray, LMJL, F-44000 Nantes, France}
\email{joackim.bernier@univ-nantes.fr}

\author[R. Häberli]{Ramona Häberli$^2$}
\author[G. Vilmart]{Gilles Vilmart$^2$}
\address{$^2$Université de Genève, Section de mathématiques, 7-9 rue du Conseil-Général, CH-1211 Genève 4, Schwitzerland}
\email{ramona.haeberli@unige.ch}
\email{gilles.vilmart@unige.ch}


\begin{abstract} 
Splitting methods are a natural choice for the numerical time integration of partial differential equations, and arbitrary high order splitting schemes exist for Schrödinger equations with periodic boundary conditions. However, in the presence of non-periodic boundary conditions, we show that they suffer in general from an order reduction, even for smooth initial conditions. The reason for such order reduction phenomena are so-called compatibility conditions, which are not preserved by classical splitting schemes. In this paper, we introduce a family of modified splitting methods for one-dimensional linear Schrödinger equations with homogeneous Dirichlet boundary conditions, which achieve an arbitrary high order, and do not suffer from any order reduction. This is illustrated with a fourth order splitting scheme considering initial conditions with various regularity properties.
\end{abstract}

\maketitle

\setcounter{tocdepth}{1} 

\section{Introduction}
\sectionmark{\footnotesize Introduction}

We aim at studying smooth solutions to the linear Schrödinger equation 
\begin{equation} \label{eq:ev_Dir}
 \ic \partial_t u = (\partial_x^2+V) u \quad \text{in}\; \R \times(0,1), \qquad  u(\cdot,0) = u(\cdot,1)=0 \quad \text{in} \;\R,
\end{equation}
where $V\in C^\infty([0,1];\mathbb{R})$ is a given potential and $u :  \mathbb{R} \times [0,1] \to \mathbb{C}$. More precisely, we look for solutions 
\begin{equation} \label{cond:reg}
u\in C^0(\mathbb{R};H^{2k}) \cap  C^1(\mathbb{R};H^{2k-2}) 
\end{equation}
for some $k\geq 1$, where the Sobolev spaces are defined as usual by
$$
H^{k} = \{ u \in L^2([0,1];\mathbb{C}) \ | \ \forall \ell \leq k, \; \partial_x^\ell u \in L^2([0,1];\mathbb{C})  \}.
$$
In general, Schrödinger equations serve as the fundamental evolution model for driven quantum systems, e.g. in optics~\cite{Sulem:1999:TNS}, quantum fluids~\cite{Bao:2005:FTL, Bao:2008:AGP}, and laser-matter interaction such as strong-field ionization~\cite{Wells:2019:AFA}. The one-dimensional linear problem~\eqref{eq:ev_Dir} is of use to linearize nonlinear problems and approximates higher-dimensional confined quantum systems, especially in the context of splitting methods. Thereby, the problem is often studied over a torus, i.e. with periodic boundary conditions~\cite{Blanes:2000:SMF,Faou:2012:GNI,Lubich:2008:FQT}, see also~\cite{Ji:2024:LRF, Ji:2025:LRE}, where low regular initial data is considered. In contrast, on an unbounded domain, usually one inserts absorbing boundary conditions for the time integration of the problem~\cite{Antoine:2008:ARO, Antoine:2001:CSA, Antoine:2012:ABC, Antoine:2013:ABC}.
However, when modeling a particle confined to a bounded region, we impose homogeneous Dirichlet boundary conditions.

A common way to integrate in time the boundary value problem~\eqref{eq:ev_Dir} on a time inverval $[0, T]$ for some $T > 0$, is to approximate the exact flow by a numerical scheme at time $t_n = n\tau$, where $\tau >0$ is the time step. We write 
$$
u_{n+1} = \Phi_{\tau}u_n  \approx u(t_{n+1}) =  e^{-\ic \tau (\partial^2_x + V)}u(t_n), \quad n=0, 1, \ldots
$$
where $\Phi_{\tau}$ is one step of a numerical scheme, e.g. a splitting method, which approximates the exact flow $e^{-\ic \tau (\partial^2_x + V)}$. We consider two subproblems, namely
\begin{equation} \label{prob:lapl}
\ic \partial_t u(t,x) = \partial_x^2 u(t,x) \quad \text{in}\; \R \times(0,1), \qquad  u(t,0) = u(t,1)=0  \quad \text{in} \;\R,
\end{equation}
as well as the potential equation
\begin{equation} \label{prob:pot}
\ic \partial_t u(t,x) = V(x)u(t,x) \quad \text{in}\;  \R \times(0,1),
\end{equation}
where we denote the exact flows by $e^{-\ic t \partial^2_x}u^{(0)}$ and $e^{-\ic t V}u^{(0)}$ for given initial data $u(0,\cdot)=u^{(0)}$ and $t \in (0, T]$. Although an operator splitting is not absolutely necessary for the time integration of the linear problem~\eqref{eq:ev_Dir}, splitting schemes allow a separate treatment of each sub-operator, which can lead to increased efficiency and easier implementation. 

In~\cite{Jahnke:2000:EBF} (see also~\cite{Bao:2024:OEB} for lower regularity assumptions), second order convergence for the Strang splitting scheme
\begin{equation} \label{strang}
\Phi^{\text{Strang}}_{\tau} = e^{-\frac{\ic \tau}{2}V}e^{-\ic \tau\partial^2_x}e^{-\frac{\ic \tau}{2}V}
\end{equation}
is shown when integrating $\ic \partial_t u = (\partial_x^2+V) u$ over the torus, i.e. in the context of periodic boundary conditions. More general, in~\cite{Thalhammer:2008:HOE} splitting methods
\begin{equation} \label{splitting}
\Phi_{\tau} = e^{-\ic b_s\tau\partial^2_x}e^{-\ic a_s\tau V} \cdots e^{-\ic b_1\tau\partial^2_x}e^{-\ic a_1\tau V},
\end{equation}
of arbitrarily high order are achieved for real coefficients $a_j, b_j, j=1, \ldots, s, s \geq 1$, satisfying particular order conditions. While the full order convergence seems to persist for the second order scheme~\eqref{strang} when integrating in time the initial problem~\eqref{eq:ev_Dir} with homogeneous Dirichlet boundary conditions, splitting methods of higher order suffer in general from an order reduction. For the numerical experiments, we focus on the symmetric splitting scheme
\begin{align} \label{y4}
\Phi^{\text{Y0}}_{\tau} = e^{-\ic a_1\tau V} e^{-\ic b_1\tau\partial^2_x}e^{-\ic a_2\tau V}e^{-\ic b_2\tau\partial^2_x}e^{-\ic a_2\tau V} e^{-\ic b_1\tau\partial^2_x}e^{-\ic a_1\tau V},
\end{align}
introduced by Yoshida~\cite{Yoshida:1990:COH} in the context of Hamiltonian systems, with real coefficients 
$$
a_1 = (2(2-2^{1/3}))^{-1}, \quad a_2 = (1-2^{1/3})a_1, \quad  b_1 = 2a_1, \quad b_2 = -2^{4/3}a_1,
$$
based on the order conditions in~\cite{Neri:1985:LAA}. Method \eqref{y4} is of formal order four, i.e. we observe a fourth order convergence of the numerical solution $u_n = (\Phi^{Y0}_{\tau})^n u^{(0)}$ to non-stiff ordinary differential equations with smooth vector fields. However, in general it converges with order strictly smaller than four when applied to problem~\eqref{eq:ev_Dir}. 
\begin{figure}[!tbp]
\begin{minipage}[c]{0.48\textwidth}
\includegraphics[width=\textwidth]{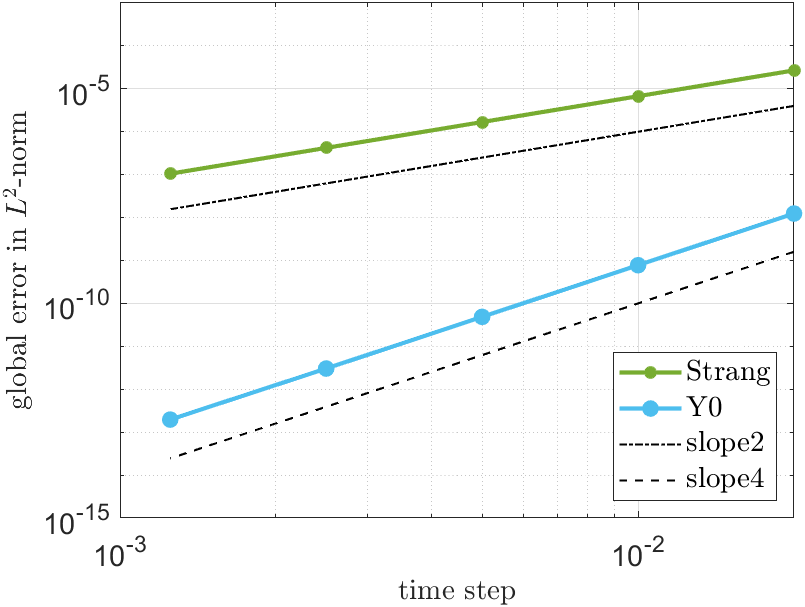}
\centering \small
odd initial condition, even potential
\end{minipage}
\quad
\begin{minipage}[c]{0.48\textwidth}
\includegraphics[width=\textwidth]{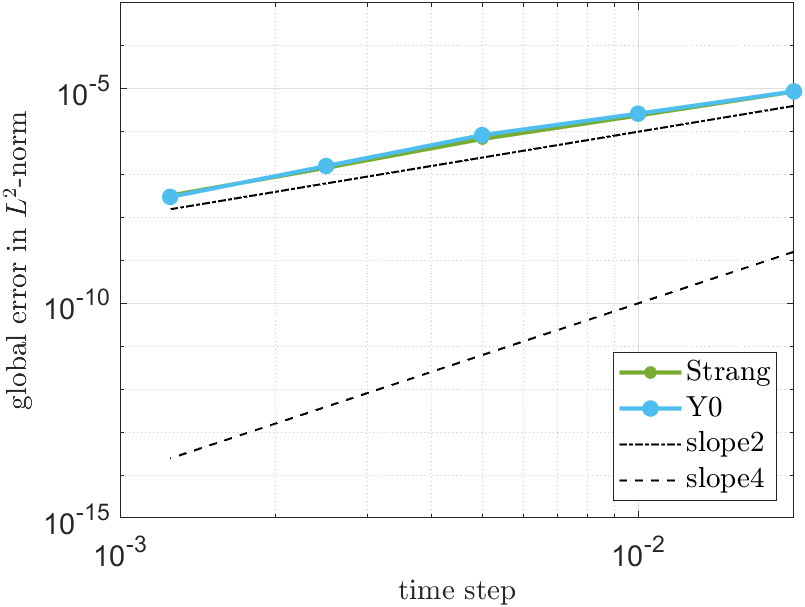}
\centering \small
more general initial condition and potential
\end{minipage}
\vspace*{1mm}
\caption{\small Convergence error of the splitting schemes~\eqref{strang} and~\eqref{y4} applied to problem~\eqref{eq:ev_Dir} for initial condition and potential $u^{(0)}(x)=\sin(2\pi x), V(x)=\cos(2\pi x)$ (left) and $u^{(0)}(x)=x(1-x)e^{x(1-x)}, V(x)=1+4x-4x^2$ (right) respectively. Reference slopes of order two and four are given in dashed lines. We set the final time $T=0.1$, and use $N=512$ points to discretize the interval $(0,1)$. For the reference solution, we consider an explicit fourth order Runge-Kutta method with time step $\tau = 10^{-6}$.}
\label{fig:naiv}
\end{figure}
In Figure~\ref{fig:naiv}, we illustrate the convergence behavior of the Strang splitting method~\eqref{strang} and the classical scheme Y0~\eqref{y4}. On the left panel, for the initial condition $u^{(0)}(x)=\sin(2\pi x)$ and the potential $V(x)=\cos(2\pi x)$, we observe full order convergence for both schemes. The choice of the even potential and the odd initial data is somehow favorable, since the numerical solutions of~\eqref{strang} and~\eqref{y4} correspond to the splitting solutions of problem~\eqref{eq:ev_Dir}, when considered with periodic boundary conditions, see Remark~\ref{rem:oddeven} for more details. Moreover, for $u^{(0)}(x)=x(1-x)e^{x(1-x)}$ and $V(x)=1+4x-4x^2$, this property is lost, and~\eqref{y4} converges with reduced order two. Thus, both in terms of the convergence order and the error constant, the naive method~\eqref{y4} does not perform better than the second order scheme~\eqref{strang}. In the literature~\cite{Bertoli:2021:SOT, Bertoli:2020:SSM,Einkemmer:2015:OOR, Einkemmer:2016:OOR,Haeberli:2026:OTO}, there exist modifications of splitting schemes, which avoid such order reduction phenomena for parabolic problems. However, in general there is no improvement in terms of the order when we directly apply those correction techniques in the context of Schrödinger equations.

The reason for such an order reduction as observed in Figure~\ref{fig:naiv} is caused by the presence of some non-trivial generalized boundary conditions. We show that whenever the solution to problem~\eqref{eq:ev_Dir} is smooth in the sense of~\eqref{cond:reg}, it satisfies so-called \emph{compatibility conditions} (see Proposition~\ref{prop:compat_cond}). In other words, $u$ lives in the function space
\begin{equation} \label{cond:comp}
H^{2k}_{V} = \{ u \in H^{2k} \ | \ \forall \ell =0,1, \ldots, k-1,\ (\partial_x^{2} + V)^\ell u(0) =  (\partial_x^{2} + V)^\ell u(1)=0\}.
\end{equation}
In~\cite{Kato:1985:ADE}, such conditions were first introduced for non-autonomous problems $u_t +A(t)u = f(t)$, where $A$ is a closed operator with good stability properties, and later applied to a generalized Schrödinger equation with Neumann/Dirichlet boundary conditions. We also cite~\cite{Kawashima:1992:GEA}, where the same framework is used to show global existence of small solutions to nonlinear viscoelasticity. Furthermore, compatibility conditions were studied in order to achieve high regularity of solutions to hyperbolic problems with non-homogeneous boundary conditions~\cite{Lasiecka:1986:NHB, Perrin:2025:COR}.

If we integrate in time problem~\eqref{eq:ev_Dir} by means of classical splitting methods, the compatibility conditions are in general not preserved, which may lead to a loss of regularity. In particular, the operator $u \mapsto Vu$ does not map $H^{2k}_V$-functions to $H^{2k}_V$. For instance, for $k=2$, there holds that $H^4_V = H^4_\mathrm{Dir}$, where we denote in general by
\begin{equation} \label{cond:comp0}
H^{2k}_{\mathrm{Dir}} = \{ u \in H^{2k} \ | \ \forall \ell =0,1, \ldots, k-1, \; \partial_x^{2\ell} u(0) =  \partial_x^{2\ell} u(1)=0 \}
\end{equation}
the space~\eqref{cond:comp} for $V=0$. Then, considering $u(x)=\sin(2\pi x)$, we have $u \in H^{2k}_{\mathrm{Dir}}$ for all $k \geq 0$. However, if $u$ is multiplied by $V(x)=\sin(2\pi x)$, there is $\partial^{2}_x (Vu)(y)\neq 0$, $y \in \{0,1\}$ (and analogously for higher derivatives), and consequently $Vu \notin H^{4}_{\mathrm{Dir}}$, see also Figure~\ref{fig:intro_schr_fm} in Section~\ref{sec:numerics}. This phenomenon does not appear when periodic boundary conditions are considered. That is why the convergence results for high order splittings over a torus cannot directly be applied in the context of zero boundary conditions. In general, in order to get high order convergence, the solution has to be sufficiently smooth, which is the reason why the compatibility conditions intervene.

\medskip

In this paper, we introduce a family of numerical methods, which integrate in time the evolution problem~\eqref{eq:ev_Dir} with arbitrarily high order.
Our approach relies on writing the smooth solution as
\begin{align} \label{def:umod}
u(t) = e^{\mathfrak{C}_k} e^{-\ic t(\partial_x^2+V^{\text{cor}}_k)} e^{-\mathfrak{C}_k}u^{(0)}.
\end{align}
Here, $v=e^{-\mathfrak{C}_k}u$ can be seen as a change of variable $H^{2k}_{V} \to H^{2k}_{\mathrm{Dir}}$, where we define corrector functions $\mathfrak{C}_k$, $k \geq 1$, in a way such that $v$ and all its derivatives $\partial^{2\ell}_xv$, $\ell < k$, vanish at the boundary. Furthermore, $ e^{-\ic t(\partial_x^2+V^{\text{cor}}_k)}$ denotes the exact flow to the modified problem
\begin{equation} \label{eq:cauchy_intro}
 \ic \partial_t v= (\partial_x^2+V^{\text{cor}}_k) v, \quad \text{in}\; \R \times(0,1) \qquad  v(\cdot,0) = v(\cdot,1)=0 \quad \text{in}\; \R,
\end{equation}
where $V^{\text{cor}}_k=e^{- \mathfrak{C}_k} (\partial_x^2 + V) e^{ \mathfrak{C}_k  }- \partial_x^2$ defines the modified potential (which is no more a multiplication operator). We will show that this flow preserves the function space $H^{2k}_{\mathrm{Dir}}$.
For the time discretization, we use splitting methods to approximate the solution of problem~\eqref{eq:cauchy_intro}. For $\tau >0$, we denote by $\Phi^{\text{cor}}_{\tau}$
one step of a splitting scheme~\eqref{splitting} applied to the Cauchy problem. This corrected splitting schemes preserve the space regularity, e.g. the modified compatibility conditions, given by $H^{2k}_{\mathrm{Dir}}$. In order to discretize in space, we first apply a periodic extension to the torus of double size of the interval $[0,1]$, and define the respective functions on this torus, we write e.g. $w^{(0)}$ for the extended initial condition. Then, we use pseudo-spectral methods in order proceed the space discretization, where the index $N$ indicates that the operators and functions are discretized in space, and corresponds to the number of points we use to discretize the space $(0,1)$, to the number of Fourier modes respectively. Finally, for the numerical schemes
\begin{align} \label{splittingmod}
u_{n,N} = e^{\mathfrak{C}_{k,N}} (\Phi^{\text{cor}}_{\tau, N})^n  e^{-\mathfrak{C}_{k,N}}w_N^{(0)}
\end{align}
there holds the following error estimate.
If $1\leq a \leq k$ and $u^{(0)} \in H^{2k}_V$ then for all $T>0$ and $N\geq 1$, we have 
\begin{equation*} 
\| u(t_n) - u_{n,N}\|_{H^{2a}} \lesssim_{T,k} (\tau^{\min(k-a,p)} + N^{-2(k-a)}) \| w^{(0)} \|_{H^{2k}}\quad \text{for}\; 0 \leq t_n \leq T,
\end{equation*}
where $u \in C^0([0,T];H^{2k}_V) $ is the solution of \eqref{eq:ev_Dir} and $p \geq 1$ is the formal order of the splitting scheme $\Phi_{\tau}$.

The outline of the paper is as follows. In Section~\ref{sec:comp}, we define the compatibility conditions and show that~\eqref{def:umod} is a unique solution~\eqref{cond:reg} to~\eqref{eq:ev_Dir}. Furthermore, we discretize this solution in time by means of classical splitting schemes. In Section~\ref{sec:discret}, we explain the space discretization and introduce the notation in order to prove the main result in a fully discretized setting. In Section~\ref{sec:convanal}, we show that the modified splitting schemes~\eqref{splittingmod} integrate in time problem~\eqref{eq:ev_Dir} with arbitrary high order, and therefore avoid order reduction. Finally, in Section~\ref{sec:numerics} we confirm numerically the convergence result while applying the presented modifications to the fourth order splitting scheme~\eqref{y4}.

\section{Compatibility conditions and correctors} \label{sec:comp}
\sectionmark{\footnotesize Compatibility conditions and correctors}

In this section we give a representation of smooth solutions~\eqref{cond:reg} to the boundary value problem~\eqref{eq:ev_Dir}, and discretize them in time by means of splitting methods. Then we show arbitrarily high order convergence of the new schemes.
\subsection{Compatibility conditions} In the following result, which is inspired by~\cite{Kato:1985:ADE}, we show that smooth solutions to \eqref{eq:ev_Dir} satisfy the compatibility conditions defined in~\eqref{cond:comp}.

\begin{proposition} \label{prop:compat_cond}Let $k\geq 1$ and $u\in C^0(\mathbb{R};H^{2k}) \cap  C^1(\mathbb{R};H^{2k-2}) $ be a solution to \eqref{eq:ev_Dir} then
$$
\forall j\leq k, \quad u\in C^j(\mathbb{R}; H^{2k-2j}_V)
$$
where
$$
\forall a\geq 0, \quad H^{2a}_{V} = \{ u \in H^{2a} \ | \ \forall \ell< a,\ (\partial_x^{2} + V)^\ell u(0) =  (\partial_x^{2} + V)^\ell u(1)=0\}.
$$
\end{proposition}
\begin{proof}
In a first step, we show the result for $j=0$, i.e. $u \in C^0(\mathbb{R}; H^{2k}_V)$, by induction over $k$. Note that
$$
H^2_V = \{ u \in H^{2} \ | \ u(0) =  u(1)=0\}= H^2_{\mathrm{Dir}},
$$
therefore, $u \in C^0(\mathbb{R};H^{2}_V)$
and the result holds true for $k=1$. Then, we consider a solution $u\in C^0(\mathbb{R};H^{2k}) \cap  C^1(\mathbb{R};H^{2k-2})$ of \eqref{eq:ev_Dir}. Since $H^{2k}\subset H^{2k-2}$ continuously, by induction hypothesis we have $u\in  C^0(\mathbb{R};H^{2k-2}_V)$. Thus, since $H^{2k-2}_V$ is a closed subset of $H^{2k-2}$ and  $u\in C^1(\mathbb{R};H^{2k-2})$, we deduce that $u\in C^1(\mathbb{R};H^{2k-2}_V)$. Finally, since $u$ is a solution of \eqref{eq:ev_Dir}, there holds true that $(\partial_x^2+V)u\in C^0(\mathbb{R};H^{2k-2}_V)$ and hence, $u\in C^0(\mathbb{R};H^{2k}_V)$.

In a second step, for $k \geq 1$, we show that if $u \in  C^0(\mathbb{R}; H^{2k}_V)\cap C^1(\mathbb{R}; H^{2k-2}_V)$ is a solution of \eqref{eq:ev_Dir}, then $u\in C^j(\mathbb{R}; H^{2k-2j}_V)$ for all $j \leq k$. Indeed, by a straightforward induction on $j$, we have that
$u\in C^j(\mathbb{R}; H^{2k-2j})$ for all $j\leq k$ with $\partial_t^j u = (-\ic)^j (\partial_x^2+V)^ju$. Thus, using $u \in  C^0(\mathbb{R}; H^{2k}_V)$, we deduce that $\partial_t^j u \in C^0(\mathbb{R}; H^{2k-2j}_V)$. Finally, noticing that $ H^{2k-2i}_V \subset H^{2k-2j}_V$ for $i\leq j$, we get, as expected, $u\in C^j(\mathbb{R}; H^{2k-2j}_V)$.
\end{proof}

Furthermore, we exhibit a polynomial structure behind these compatibility conditions. 

\begin{lemma} \label{lem:appendix} There exists a family of polynomials $f_{i} \in C^\infty([0,1];\mathbb{R})[X]$, $i\geq 2$, with $\deg f_i = i-1$, such that
$$
\forall \ell \geq 0, \quad (\partial_x^2+V)^{\ell} =\partial_x^{2\ell} + \sum_{i=2}^{2\ell} f_{i}(\ell) \partial_x^{2\ell-i}.
$$
\end{lemma}
\begin{proof} Expanding $(\partial_x^2+V)^{\ell}$, it is clear that there exists $C^\infty$ functions $g_{\ell,i}$ such that
$$
\forall \ell \geq 0, \quad (\partial_x^2+V)^{\ell} =\partial_x^{2\ell} + \sum_{i=2}^{2\ell} g_{\ell,i} \partial_x^{2\ell-i}.
$$
What we really have to prove is that $g_{\ell,i}$ is polynomial with respect to $\ell$. So we are going to derive an induction relation for these functions.

Indeed, for all $\ell\geq 0$, we have
\begin{equation*}
\begin{split}
&(\partial_x^2+V)^{\ell+1}\\ &= (\partial_x^2 + V) (\partial_x^2+V)^{\ell} \\
&=  (\partial_x^2 + V) \partial_x^{2\ell} + \sum_{i=2}^{2\ell} ((\partial_x^2+V) g_{\ell,i})  \partial_x^{2\ell-i} +  g_{\ell,i}  \partial_x^{2\ell+2-i} +  2 (\partial_x g_{\ell,i} ) \partial_x^{2\ell+1-i}\\
&= (\partial_x^2 + V) \partial_x^{2\ell} + \sum_{i=4}^{2\ell+2} ((\partial_x^2+V) g_{\ell,i-2})  \partial_x^{2\ell+2-i} + \sum_{i=2}^{2\ell}  g_{\ell,i}  \partial_x^{2\ell+2-i} + \sum_{i=3}^{2\ell+1} 2 (\partial_x g_{\ell,i-1} ) \partial_x^{2\ell+2-i}.
\end{split}
\end{equation*}
For the special case $i=2$, we get
$$
 g_{\ell+1,2} = g_{\ell,2} + V,
$$
which implies
$$
 \forall \ell \geq 0, \quad g_{\ell,2} = \ell V.
$$
Additionally, for $i=3$, we obtain
$$
 g_{\ell+1,3} = g_{\ell,3} + 2 \partial_x g_{\ell,2} , 
$$
and consequently,
$$
 \forall \ell \geq 0, \quad g_{\ell,3}= \sum_{k=0}^{\ell-1} 2 \partial_x g_{k,2} = 2 \partial_x V  \sum_{k=0}^{\ell-1} k = (\ell-1)\ell \partial_x V . 
$$
Generally, we set $g_{\ell,i}=0$ if $i\notin \{2,\cdots,2\ell\}$ and we deduce that
\begin{equation}
\label{eq:g1}
\forall 4\leq i\leq 2\ell, \quad g_{\ell+1,i} = ((\partial_x^2+V) g_{\ell,i-2})  + g_{\ell,i} + 2 \partial_x g_{\ell,i-1}.
\end{equation}
Moreover, the following identities hold true,
$$
 g_{\ell+1,2\ell+2} =   (\partial_x^2 + V)  g_{\ell,2\ell} \quad \text{and} \quad
 g_{\ell+1,2\ell+1}  =  (\partial_x^2 + V)  g_{\ell,2\ell-1} + 2 \partial_x  g_{\ell,2\ell}.
$$

Finally, by \eqref{eq:g1} and $i\geq 4$, we have
\begin{equation*}
\begin{split}
 \forall \ell \geq \lceil i/2 \rceil +1, \quad g_{\ell,i} &=g_{\lceil i/2 \rceil,i}  + \sum_{k= \lceil i/2 \rceil}^{\ell-1} g_{k+1,i} - g_{k,i} \\
 &= g_{\lceil i/2 \rceil,i}  + \sum_{k= \lceil i/2 \rceil}^{\ell-1}   ((\partial_x^2+V) g_{k,i-2})   + 2 \partial_x g_{k,i-1}
\end{split}
\end{equation*}
So by induction, it is clear that $g_{\ell,i}$ is a polynomial with respect to $\ell$ of degree $i-1$. 
\end{proof}

Finally, by means of Lemma~\ref{lem:appendix}, we reformulate the compatibility conditions and represent them in terms of the computed polynomials.

\begin{proposition} \label{prop:comp_prec} For $y\in \{0,1\}$, there exists a family of polynomials $(P_{y,j})_{j\geq 1} \subset \mathbb{R}[X]$, with $\mathrm{deg} P_{y,j} \leq 2j$ such that for all $a\geq 0$, 
\begin{equation}
\label{eq:comp_prec}
H^{2a}_{V} = \Big\{ u \in H^{2a} \ | \ \forall \ell < a,\forall y \in \{0,1\}, \ \big(\partial_x^{2\ell} - \sum_{j=1}^{\ell-1} P_{y,j}(\ell)  \partial_x^{2\ell-(2j+1)} \big) u(y) = 0 \Big\}.
\end{equation}
\end{proposition}

\begin{proof} Without loss of generality, we focus on the left boundary $y=0$ and, for readability, we omit the subscript $y$. We prove~\eqref{eq:comp_prec} by induction over $a\geq 0$. For the initialization, it suffices to note that 
$$
H^{0}_{V}=  H^{0}_{\mathrm{Dir}}, \quad H^{2}_{V}=  H^{2}_{\mathrm{Dir}} \quad \mathrm{and} \quad H^{4}_{V}=  H^{4}_{\mathrm{Dir}}.
$$
Therefore, let $a \geq 3$ and consider $u \in H^{2a}_{V}$. By Lemma \ref{lem:appendix}, we have
\begin{equation}
\label{eq:Q}
(\partial_x^2+V)^\ell u(0)  =   \partial_x^{2\ell}u(0) - \sum_{i=2}^{2\ell} Q_{i}(\ell) \partial_x^{2\ell-i} u(0),
\end{equation}
where $Q_{i} \in \mathbb{R}[X]$ is defined by $Q_{y,i}(\ell) := -f_i(\ell)(y)$ and satisfies $\mathrm{deg}\, Q_{i} \leq i-1$.

Using that $u(0)=\partial_x^2 u(0) = \partial_x^4 u(0) = 0$, we deduce of \eqref{eq:Q} and the fact that $ (\partial_x^2+V)^\ell u(0) =0$ for $\ell = 0, \ldots, a-1$,
$$
\partial_x^{2\ell}u(0) - \sum_{i=1}^{\ell-2} Q_{2i}(\ell) \partial_x^{2\ell-2i} u(0) -  \sum_{i=1}^{\ell-1} Q_{2i+1}(\ell) \partial_x^{2\ell-(2i+1)} u(0)  =0.
$$
Then, we define the polynomials $P_i \in \mathbb{R}[X]$, $i\geq 1$, by means of the recursion formula
\begin{equation} \label{def:P}
P_{i}(\ell) =  Q_{2i+1}(\ell) + \sum_{m=1 }^{i-1}  Q_{2m}(\ell)   P_{i-m}(\ell-m) .
\end{equation}
Note that $\mathrm{deg}\, P_i \leq 2i$. Then, by induction we have $ \partial_x^{2\ell-2i} u(0) = \sum_{j=1}^{\ell-i-1} P_{j}(\ell-i)  \partial_x^{2\ell-2i-(2j+1)} u(0) $ and thus, we get
$$
\partial_x^{2\ell}u(0) - \sum_{i=1}^{\ell-2} Q_{2i}(\ell)  \sum_{m=1}^{\ell-i-1} P_{m}(\ell-i)  \partial_x^{2\ell-2i-(2m+1)} u(0) -  \sum_{i=1}^{\ell-1} Q_{2i+1}(\ell) \partial_x^{2\ell-(2i+1)} u(0)  =0.
$$
Re-ordering these sum, we conclude that
$$
\partial_x^{2\ell}u(0) -  \sum_{i=1}^{\ell-1} \Big( \underbrace{  Q_{2i+1}(\ell) + \sum_{m=1 }^{i-1}   Q_{2m}(\ell)   P_{i-m}(\ell-m) }_{= P_{i}(\ell) } \Big) \partial_x^{2\ell-(2i+1)} u(0)  =0,
$$
which justifies definition~\eqref{def:P} of $P_i(\ell)$ and concludes the proof.
\end{proof}
The polynomial representation of the compatibility conditions in~\eqref{eq:comp_prec} is crucial for the construction of the corrector functions.


\subsection{Change of variable in terms of correctors} \label{sec:corr}

Throughout this paper, for $n\geq 0$ and $y\in~\{0,1\}$, we consider a family of smooth functions $ \varphi_{n,y} \in C^\infty([0,1];\mathbb{R})$ satisfying
$$
\forall z\in \{0,1\},\forall i\geq 0, \quad \partial_x^{i}    \varphi_{n,y} (z) = \mathds{1}_{z=y} \mathds{1}_{i=n}.
$$
Moreover, in order to construct the correctors, we introduce the following operators.

\begin{definition} \label{def:integral} Given $n\geq 0$ and $y\in \{0,1\}$, the operator $\mathcal{I}_{n,y}$ is defined inductively by
$$
\forall u\in L^1,\forall x\in [0,1], \quad \mathcal{I}_{n+1,y}u(x) = \int_y^x  (\mathcal{I}_{n,y}u)(z) \, \mathrm{d}z \quad \mathrm{and} \quad   \mathcal{I}_{0,y} u = u.
$$
\end{definition}

\begin{remark} An alternative concise way to define $\mathcal{I}_{n,y}$ for $n\geq 1$ is
$$
\mathcal{I}_{n,y} u(x) = \frac1{(n-1)!} \int_y^x (x-z)^{n-1}  u(z) \mathrm{d}z.
$$
\end{remark}
In the following proposition, we construct a family of correctors in terms of the smooth functions $\varphi_{n,y}$ and the operators $\mathcal{I}_{n,y}$.
\begin{proposition} \label{prop:constr_correct} There exist real coefficients $\alpha_{i,n,y} \in \mathbb{R}$, with $i\geq 1$ and $1\leq n \leq 2i+1$, such that for all $ k \geq 1$, the correctors defined by
\begin{equation} \label{def:ck}
\mathfrak{C}_k u = \sum_{y\in \{0,1\}} \sum_{i=1}^{k-1} \sum_{n=1}^{2i+1} \alpha_{i,n,y}  \varphi_{2i+1-n,y} \mathcal{I}_{n,y} u
\end{equation}
satisfy
$$
 \exp(\mathfrak{C}_k)  H^{2k+2}_{\mathrm{Dir}} \subset  H^{2k+2}_{V}.
$$
\end{proposition}

We first prove Proposition~\ref{prop:constr_correct} in the model case $k=2$, the general case then follows.

\medskip

\begin{ProofOf}{Proposition~\ref{prop:constr_correct} for $k=2$}
First, we note that for all real coefficients $\alpha$, the corrector $\mathfrak{C}_2$ maps $H^{6}$ into itself. Moreover, due to Proposition \ref{prop:comp_prec} we have that  $ \exp(\mathfrak{C}_2) H^{6}_{\mathrm{Dir}} \subset  H^{6}_{V}$ if and only if for $y\in \{0,1\}$,
\begin{align*}
\forall u\in H^{2}_{\mathrm{Dir}}, \quad e^{\mathfrak{C}_2} u(y)=0, \\
\forall u\in H^{4}_{\mathrm{Dir}}, \quad \partial_x^{2} e^{\mathfrak{C}_2} u(y)=0, \\
\forall u\in H^{6}_{\mathrm{Dir}}, \quad \big(\partial_x^{4} - P_{y,1}(2)  \partial_x \big) e^{\mathfrak{C}_2} u(y)=0. 
\end{align*}
The first two equations are satisfied anyway, so lets concentrate on the third one. The point is just to rewrite this property as a system of equations on $\alpha$ and finally to solve it. Thus, since $ \exp(\mathfrak{C}_2) =  \sinh(\mathfrak{C}_2) +  \cosh(\mathfrak{C}_2)$, the equations we have to solve are
\begin{align} \label{eq:eventerms}
\partial_x^{4}  \cosh(\mathfrak{C}_2) u(y) = P_{y,1}(2)  \partial_x \sinh(\mathfrak{C}_2) u(y),
\end{align}
for $y \in \{0, 1\}$, what corresponds to the even terms. Similarly, for the odd terms we get
\begin{align} \label{eq:oddterms}
\partial_x^{4}  \sinh(\mathfrak{C}_2) u(y) = P_{y,1}(2)  \partial_x \cosh(\mathfrak{C}_2) u(y).
\end{align}
Now, we note that, for $ n = 1, 2, 3$ we have
\begin{align*}
\partial_x^p  \big( \varphi_{3-n,y}  \mathcal{I}_{n,y} u  \big)(y)
=\left\{ \begin{array}{cll} \displaystyle  \binom{4}{3-n}  \partial_x u(y) & p=4 \\
0 & p=1, 2, 3. \end{array}\right.
\end{align*}
More generally, given $q\geq 2$, $y_a\in \{0,1\}$ and $1 \leq n_a \leq 2i_a+1$, $i_a\geq 1$  we have
\begin{equation*}
\partial_x^4  \big( \prod_{1\leq a \leq q}^{\rightarrow} \varphi_{2i_a+1-n_a,y_a}  \mathcal{I}_{n_a,y_a} \big) u(y) =0.
\end{equation*}
Thus, the terms given in~\eqref{eq:eventerms} are equal to zero, whereas for the odd terms in~\eqref{eq:oddterms} we get
\begin{equation*}
\partial_x^{4}  \sinh(\mathfrak{C}_2) u(y) = \partial_x^{4} \mathfrak{C}_2 u(y) =  (6\alpha_{1,1,y} +  4\alpha_{1,2,y} + \alpha_{1,3,y})  \partial_x u(y),
\end{equation*}
and
\begin{equation*}
 P_{y,1}(2)  \partial_x\cosh(\mathfrak{C}_2) u(y) =  P_{y,1}(2)  \partial_x u(y).
\end{equation*}
As a consequence, the system we have to solve is for $y\in \{0,1\}$,
\begin{equation} \label{eq:alpha_k2}
6\alpha_{1,1,y} +  4\alpha_{1,2,y} + \alpha_{1,3,y}=  P_{y,1}(2) = -2\partial_xV(y),
\end{equation}
and thus, there exist coefficients $\alpha_{1,n,y}, 1 \leq n \leq 3$, such that $\exp(\mathfrak{C}_2) u \in  H^{6}_{V}$ for any $u \in H^{6}_{\mathrm{Dir}}$.
\end{ProofOf}

\begin{ProofOf}{Proposition~\ref{prop:constr_correct} for general $k > 2$}
First, we note that whatever the coefficients $\alpha$, the correctors $\mathfrak{C}_k$ maps $H^{2k+2}$ into itself. It suffices to check the compatibility conditions. More precisely, by Proposition \ref{prop:comp_prec}, we have that  $ \exp(\mathfrak{C}_k)   H^{2k+2}_{\mathrm{Dir}} \subset  H^{2k+2}_{V}$ if and only if
$$
\forall y\in \{0,1\},\forall \ell\leq k,\forall u\in H^{2\ell+2}_{\mathrm{Dir}}, \quad \big(\partial_x^{2\ell} - \sum_{j=1}^{\ell-1} P_{y,j}(\ell)  \partial_x^{2\ell-(2j+1)} \big) e^{\mathfrak{C}_k} u(y)=0.
$$
Then, the point is just to rewrite this property as a system of equations on $\alpha$ and finally to solve it. So from now, without loss of generality, we consider $u\in  H^{\infty}_{\mathrm{Dir}} = \bigcap_{k \geq 1} H^{2k}_{\mathrm{Dir}} $, $y\in \{0,1\}$ and $\ell\leq k$. Now, we note that, for all $1 \leq n \leq 2i+1$, $i\geq 1$ we have
\begin{equation}
\label{eq:utile}
\partial_x^p  \big( \varphi_{2i+1-n,y}  \mathcal{I}_{n,y} u         \big)(y) =  
 \left\{ \begin{array}{cll} \displaystyle  \binom{p}{2i+1-n}  \partial_x^{p-(2i+1)} u(y) &  \mathrm{if} \quad p\geq 2i+1 \\
0 & \mathrm{else}. \end{array}\right.
\end{equation}
More generally, given $q\geq 1$, $y_a\in \{0,1\}$, $1 \leq n_a \leq 2i_a+1$, $i_a\geq 1$ for $1\leq a \leq q$ we have
\begin{equation*}
\begin{split}
&\partial_x^p  \big( \prod_{1\leq a \leq q}^{\rightarrow} \varphi_{2i_a+1-n_a,y_a}  \mathcal{I}_{n_a,y_a} \big) u        (y)\\ &= \left\{ \begin{array}{cll} \displaystyle \prod_{ 1 \leq a \leq q}  \binom{p- 2 i_{\leq a-1}-a+1 }{2 i_a+1-n_a}  \partial_x^{p-2i_{\leq q} -q} u(y) & \displaystyle \mathrm{if} \quad p\geq 2i_{\leq q}+q \ \mathrm{and} \quad y = \; y_1=\cdots =  y_q\\
0 & \mathrm{else} \end{array}\right.
\end{split}
\end{equation*}
where we used the notation
\begin{equation*} 
\forall b \leq q, \quad i_{\leq b} = \sum_{1\leq  a \leq b}  i_a.
\end{equation*}
Thus, since $u\in  H^{\infty}_{\mathrm{Dir}} $, the equation we have to solve is
$$
\partial_x^{2\ell}  \sinh(\mathfrak{C}_k) u(y) = \sum_{j=1}^{\ell-1} P_{y,j}(\ell)  \partial_x^{2\ell-(2j+1)}\cosh(\mathfrak{C}_k) u(y).
$$
Then we expand theses terms and using \eqref{eq:utile} and we get
\begin{equation*}
\begin{split}
&\partial_x^{2\ell}  \sinh(\mathfrak{C}_k) u(y)  \\
=& \sum_{q\geq 0} \frac1{(2q+1)!}   \sum_{\substack{1\leq i_1 \leq k-1 \\  1\leq n_1 \leq 2 i_1+1 }}  \cdots \sum_{\substack{1\leq i_{2q+1} \leq k-1 \\  1\leq n_{2q+1} \leq 2 i_{2q+1}+1 }}   \partial_x^{2\ell}   \prod_{1\leq a \leq 2q+1}^{\rightarrow} \alpha_{i_a,n_a,y} \varphi_{2i_a+1-n_a,y}  \mathcal{I}_{n_a,y}  u        (y)\\
 =& \sum_{i=1}^{\ell-1}  \partial_x^{2\ell-(2i+1)} u(y) \sum_{\substack{ i_{\leq 2q+1} + q = i \\ q\geq 0  }} \sum_{\substack{ 1\leq n_a \leq 2i_a+1 \\ \forall 1\leq a \leq 2q+1}}   \prod_{ 1 \leq a \leq 2q+1} \frac{\alpha_{i_a,n_a,y}}{a}  \binom{2\ell- 2 i_{\leq a-1}-a+1 }{2 i_a+1-n_a}.
\end{split}
\end{equation*}
and

\begin{equation*}
\begin{split}
& \sum_{j=1}^{\ell-1} P_{y,j}(\ell)  \partial_x^{2\ell-(2j+1)}\cosh(\mathfrak{C}_k) u(y)   \\
=& \sum_{q\geq 0} \frac1{(2q)!}   \sum_{\substack{1\leq i_1 \leq k-1 \\  1\leq n_1 \leq 2 i_1+1 }}  \cdots \sum_{\substack{1\leq i_{2q} \leq k-1 \\  1\leq n_{2q} \leq 2 i_{2q}+1 }}    \sum_{j=1}^{\ell-1} P_{y,j}(\ell)  \partial_x^{2\ell-(2j+1)}   \prod_{1\leq a \leq 2q}^{\rightarrow} \alpha_{i_a,n_a,y} \varphi_{2i_a+1-n_a,y}  \mathcal{I}_{n_a,y}  u        (y)\\
 =& \sum_{i=1}^{\ell-1}  \partial_x^{2\ell-(2i+1)} u(y) \sum_{\substack{i_{\leq 2q} + q +j = i \\ q\geq 0  }} \sum_{\substack{ 1\leq n_a \leq 2i_a+1 \\ \forall 1\leq a \leq 2q}}   P_{y,j}(\ell) \prod_{ 1 \leq a \leq 2q} \frac{\alpha_{i_a,n_a,y}}{a}  \binom{2\ell-2j- 2 i_{\leq a-1}-a }{2 i_a+1-n_a}.
\end{split}
\end{equation*}
In the particular case of $q=0$, we get $P_{y,i}(\ell) \partial^{2\ell-(2i+1)}_xu(y)$ for $\ell \leq k$, $i \geq 1$ and $y \in \{0, 1\}$.
As a consequence, the system we have to solve is for all $\ell \geq 0$, $i\geq 1$, $y\in \{0,1\}$
\begin{equation} \label{eq:mon_system}
\begin{split}
\sum_{1\leq n \leq 2i+1} &\alpha_{i,n,y} \binom{2\ell }{2 i+1-n} \\
&=  \sum_{\substack{ i_{\leq 2q} + q +j = i \\ q\geq 0  }} \sum_{\substack{ 1\leq n_a \leq 2i_a+1 \\ \forall 1\leq a \leq 2q}}   P_{y,j}(\ell) \prod_{ 1 \leq a \leq 2q} \frac{\alpha_{i_a,n_a,y}}{a}  \binom{2\ell-2j- 2 i_{\leq a-1}-a }{2 i_a+1-n_a} \\ 
&- \sum_{\substack{ i_{\leq 2q+1} + q = i \\ q\geq 1  }} \sum_{\substack{ 1\leq n_a \leq 2i_a+1 \\ \forall 1\leq a \leq 2q+1}}   \prod_{ 1 \leq a \leq 2q+1} \frac{\alpha_{i_a,n_a,y}}{a}  \binom{2\ell- 2 i_{\leq a-1}-a+1 }{2 i_a+1-n_a} =: \mathcal{P}_{i,y}(\ell).
\end{split}
\end{equation}
First, we note that the indices in the sums do not depend on $\ell$ and $b$ being fixed that we have
\begin{center}
$\displaystyle \ell  \mapsto \binom{2\ell}{b} = \frac{(2\ell) \cdots  (2\ell - b +1)}{ b!}$ is a polynomial of degree $b$.
\end{center}
Thus, since $P_{y,j}$ is a polynomial of degree $\mathrm{deg}\, P_{y,j} \leq 2j$, it follows that $\mathcal{P}_{i,y}$ is a polynomial of degree smaller or equal to $2i$. So we can forget the variable $\ell$ and consider \eqref{eq:mon_system} as a system of equalities between formal polynomials indexed by $i$. Then, to solve it, it suffices to note that it is triangular with non zero coefficients on the diagonal. Indeed, on the one hand $\mathcal{P}_{i,y}$ only depends on $\alpha_{b,n,y}$ with $1\leq b\leq i-1$ and $1\leq n\leq 2b+1$ and on the other hand
$$
\beta \in \mathbb{R}^{2i+1} \mapsto \sum_{1\leq n \leq 2i+1} \beta_n \frac{(2X) \cdots  (2X - (2 i+1-n) +1)}{ (2 i+1-n)!}  \in \mathbb{R}_{\leq 2i}[X]
$$
is clearly an isomorphism where $ \mathbb{R}_{\leq 2i}[X]$ denotes the space of the polynomials of degree smaller than or equal to $2i$.
\end{ProofOf}

By means of~\eqref{eq:mon_system}, we can compute the coefficients $\alpha_{i,n,y}$ in terms of the polynomials $P_{y,i}$.

\begin{table}[H]
\captionsetup{labelformat=simple,name=Table,labelfont=bf,textfont=it}
\caption{Values of $P_{y,i}(\ell)$ for $y \in \{0,1\}$ and $1 \leq i \leq 3$, evaluated in $\ell =2, 3, 4$.} \label{table:P}
\centering

\begin{tabular}{p{0.03\textwidth}|p{0.1\textwidth}|p{0.31\textwidth}|p{0.40\textwidth}}

 \toprule
 $i \backslash \ell$ & 2 & 3 & 4 \\
 \midrule
 1 &  $-2\partial_xV(y)$ & $-6\partial_xV(y)$ & $-12\partial_xV(y)$  \\

&&\\[-2.5ex]
\hline
&&\\[-2.5ex]

 2 &   $0$  & $-4\partial_x^3V(y)$   & $-24\partial_x^3V(y)$ \\

&&\\[-2.5ex]
\hline
&&\\[-2.5ex]

 3 & $0$ & $0$ & $-6\partial_x^5V(y) + 12\partial_xV(y)\partial_x^2V(y)$    \\
 \bottomrule
\end{tabular}
\end{table}

In the following example, we give possible representations of the correctors $\mathfrak{C}_2, \mathfrak{C}_3$ and $\mathfrak{C}_4$, where we make us of some values $P_{y,i}(\ell)$, see Table~\ref{table:P}.

\begin{example} \label{ex:c2c3c4}
From the system~\eqref{eq:mon_system} follows that for $k \gg i$, the coefficients $\alpha_{i,n,y}$ are uniquely defined for all $1 \leq n \leq 2i+1$ and $y \in \{0,1\}$. Otherwise, the system is underdetermined and we have a lot of degrees of freedom, hence many coefficients $\alpha_{i,n,y}$ in the definition of $\mathfrak{C}_k$~\eqref{def:ck} can be set equal to zero.
For $1 \leq n \leq 3$ and $y \in \{0,1\}$, the coefficients $\alpha_{1,n,y}$ of the operator $\mathfrak{C}_2$ satisfy the identity~\eqref{eq:alpha_k2}, so technically for each point of the boundary $y$, we can choose two of them equal to zero. Therefore, in order to achieve a solution $u\in C^0(\mathbb{R};H^{4}) \cap  C^1(\mathbb{R};H^{2}) $ to problem~\eqref{eq:ev_Dir}, we consider e.g.
\begin{equation*} 
\mathfrak{C}_2 = -1/2\partial_xV(0)\varphi_{0,0} \mathcal{I}_{2,0} -1/2\partial_xV(1)\varphi_{0,1} \mathcal{I}_{2,1}.
\end{equation*}
For $k=3$, i.e. to get $H^6$-regularity, we have to solve the systems
\begin{align*}
\alpha_{1,3,y} +  4\alpha_{1,2,y} + 6\alpha_{1,1,y} &= -2\partial_xV, \\
\alpha_{1,3,y} +  6\alpha_{1,2,y} + 15\alpha_{1,1,y} &= -6\partial_xV,
\end{align*}
and 
\begin{align*}
\alpha_{2,5,y} +  6\alpha_{2,4,y} + 15\alpha_{2,3,y} +  20\alpha_{2,2,y} + 15\alpha_{2,1,y}&= -4\partial^3_xV,
\end{align*}
for $i=1$, $i=2$ respectively. So, if we set e.g. $\alpha_{1,1,y} = \alpha_{2,1,y}  = 0$ and $\alpha_{2,n,y} = 0,  n = 3, 4, 5$ for $y=0, 1$, we get
\begin{equation*}
\mathfrak{C}_3 = \sum_{y\in \{0,1\}} (-2\partial_xV \varphi_{1,y} \mathcal{I}_{2,y} +6\partial_xV\varphi_{0,y} \mathcal{I}_{3,y}- 1/5\partial^3_xV \varphi_{3,y} \mathcal{I}_{2,y}).
\end{equation*}
Finally, for $k=4$, i.e. to achieve a solution of $H^8$-regularity, there is no degree of freedom in the case $i=1$, for $y=0, 1$ we have $\alpha_{1,1,y} = -1/2\partial_xV, \alpha_{1,2,y} = 1/4 \partial_xV$ and $\alpha_{1,3,y} = 0$. Additionally, for $i=2, 3$ we get the following systems to solve,
\begin{align*}
\alpha_{2,5,y} +  6\alpha_{2,4,y} + 15\alpha_{2,3,y} +  20\alpha_{2,2,y} + 15\alpha_{2,1,y}&= -4\partial^3_xV, \\
\alpha_{2,5,y} +  8\alpha_{2,4,y} + 28\alpha_{2,3,y} +  56\alpha_{2,2,y} + 70\alpha_{2,1,y}&= -24\partial^3_xV,
\end{align*}
and
\begin{align*}
\alpha_{3,7,y} +  8\alpha_{3,2,y} + 28\alpha_{3,5,y} +  56\alpha_{3,4,y} + 70\alpha_{3,3,y} +  56\alpha_{3,2,y} + 28\alpha_{3,1,y}&= -6\partial^5_xV+12\partial_xV\partial^2_xV.
\end{align*}
For the numerical experiments, we consider the following variant of the corrector $\mathfrak{C}_4$,
\begin{equation*} 
\begin{split} 
\mathfrak{C}_4 = \sum_{y\in \{0,1\}} (&-1/2\partial_xV \varphi_{2,y} \mathcal{I}_{1,y} +1/4\partial_xV\varphi_{1,y} \mathcal{I}_{2,y} -31/35\partial^3_xV\varphi_{3,y} \mathcal{I}_{2,y}\\
& +32/35\partial^3_xV\varphi_{2,y}\mathcal{I}_{3,y} + 1/70(-6\partial^5_xV+12\partial_xV\partial^2_xV) \varphi_{4,y} \mathcal{I}_{3,y}).
\end{split}
\end{equation*}
Note that all the derivatives $\partial^j_xV$ and function $\varphi_{n,y}$ are evaluated in $y = 0,1$.
\end{example}

Now, we establish a basic useful property on the correctors.
\begin{lemma} \label{lem:pas_si_mal} Let $ k \geq 2$ for all $j\geq 0$, we have that $\exp(\mathfrak{C}_k)$ is bounded on $H^{2j}$ and satisfies
$$
\exp(\mathfrak{C}_k) H^{2j}_{\mathrm{Dir}} =  H^{2j}_{V} \quad \mathrm{if} \quad j\leq k+1.
$$
\end{lemma}

In order to prove Lemma~\ref{lem:pas_si_mal}, we need the following result.
\begin{lemma}
\label{lem:func_ana_app} Let $E$ be a Hilbert space, $A : E\to E$ be an isomorphism and $F,G \subset E$ be two close subsets of $E$.
If $AF \subset G$ and  $\mathrm{dim}\, F^\perp \leq \mathrm{dim}\, G^\perp <\infty$, then $AF=G$.
\end{lemma}
\begin{proof} First, we note that by duality, the inclusion $AF \subset G$ is equivalent to $A^* G^\perp \subset F^\perp$. Then using that $A^*$ is bijective, we have that
$$
\mathrm{dim}\, F^\perp  \leq \mathrm{dim}\, G^\perp  =  \mathrm{dim}\, A^* G^\perp.
$$
Thus, we have $A^* G^\perp = F^\perp$, and so $AF=G$.
\end{proof}

\begin{ProofOf}{Lemma~\ref{lem:pas_si_mal}}
Since $\mathfrak{C}_k$ is bounded on $H^{2j}$, $\exp(\mathfrak{C}_k)$ is also bounded on $H^{2j}$.
By construction, we also know that $\exp(\mathfrak{C}_k)  H^{2k+2}_{\mathrm{Dir}} \subset H^{2k+2}_{V}$. Thus, if $0\leq j \leq k+1$,  by continuity of $\mathfrak{C}_k$ in $H^{2j} $ and by density, we deduce that $\exp(\mathfrak{C}_k) H^{2j}_{\mathrm{Dir}} \subset H^{2j}_{V}$. Finally, since 
$$
\mathrm{codim} \, H^{2j}_{V} = \mathrm{codim} \, H^{2j}_{\mathrm{Dir}} = 2j
$$
we conclude by a dimension argument that these inclusions are equalities, see Lemma \ref{lem:func_ana_app}.
\end{ProofOf}
Due to Lemma~\ref{lem:pas_si_mal}, the application of the operators $e^{\pm\mathfrak{C}_k}$ allows to switch between the spaces $H^{2k+2}_{\mathrm{Dir}}$ and $H^{2k+2}_{V}$.


\subsection{Modified potential} In this subsection, we define the corrected potential in terms of the corrector functions given in~\eqref{def:ck}.

\begin{definition} Given $k\geq 1$, we introduce the modified potential, defined by
$$
V^{\mathrm{cor}}_k = e^{- \mathfrak{C}_k} (\partial_x^2 + V) e^{ \mathfrak{C}_k  }- \partial_x^2.
$$
\end{definition}

\begin{theorem} \label{thm:main} Let $k\geq 1$. Then, for all $j\in \{1,\ldots,k\}$, the potential $V^{\mathrm{cor}}_k$ is bounded on $H^{2j}$, and it satisfies $V^{\mathrm{cor}}_k   H^{2j}_{\mathrm{Dir}} \subset H^{2j}_{\mathrm{Dir}}$.
\end{theorem}
\begin{proof} Let $j\in \{1,\ldots,k\}$.
In order to show that $V^{\mathrm{cor}}_k$ is bounded on $H^{2j}$, we expand the exponentials,
$$
V^{\mathrm{cor}}_k =  e^{- \mathfrak{C}_k  }  V e^{ \mathfrak{C}_k  } - \partial_x^2 + e^{-\mathrm{ad}_{\mathfrak{C}_k}}  \partial_x^2   = e^{- \mathfrak{C}_k  }  V e^{ \mathfrak{C}_k  } + \sum_{q\geq 1} \frac{(-1)^q}{q!} \mathrm{ad}_{\mathfrak{C}_k}^q \partial_x^2,
$$
where we use the notation for the iterated commutators, $\mathrm{ad}^q_{\mathfrak{C}_k}(\cdot) = [ \mathfrak{C}_k,\mathrm{ad}^{q-1}_{\mathfrak{C}_k}(\cdot)]$ for $q \geq 1$ and $\mathrm{ad}^0_{\mathfrak{C}_k} =\text{Id}$.
Since $\mathfrak{C}_k$ is bounded  on $H^{2j} $,  it suffices to prove that $[\partial_x^2 ,  \mathfrak{C}_k]$ is bounded on $H^{2j}$.
 Expanding $\mathfrak{C}_k$, it suffices to prove that for all $y\in \{0,1\}$, all $n\geq 1$ and all $\psi \in C^\infty$  the operator $[\partial_x^2 ,  \psi \mathcal{I}_{n,y}]$ is bounded on  $H^{2j}$.
Then, we notice that since $\mathcal{I}_{n,y} $ is an iterated integral,
$$
\forall n\geq 1, \forall a\geq 0,\forall y\in \{0,1\}, \quad \mathcal{I}_{n,y} \; \mathrm{is \ bounded\ from\ } H^a \ \mathrm{to} \ H^{a+n}.
$$
It follows, by composition, since $2j\geq 2$, that almost all these operators are bounded. The only remaining one is $[\partial_x^2 ,  \psi \mathcal{I}_{1,y}]$.
So, we compute it to observe a cancellation. Indeed, given $u\in H^{2j}$ and $y\in \{0,1\}$, we have
\begin{equation*}
\begin{split}
[\partial_x^2 ,  \psi \mathcal{I}_{1,y}] u &= \partial_x^2 \big( \psi \int_y^\cdot u(z) \, \mathrm{d}z \big) -  \psi \int_y^\cdot \partial_x^2 u(z) \, \mathrm{d}z \\
&= (\partial_x^2\psi) \mathcal{I}_{1,y} u + 2 (\partial_x \psi)u + \psi \partial_x u - \psi\partial_x u + \psi \partial_x u(y).
\end{split}
\end{equation*}
Since the terms of order $1$ (i.e. $\partial_x u$) cancel out and $2j\geq 2$ (which ensures that $u\mapsto \partial_x u(y)$ is bounded), we deduce that $[\partial_x^2 ,  \psi \mathcal{I}_{1,y}]$ is bounded on $H^{2j}$.

For the second property, by density, it suffices to prove that
$$
V^{\mathrm{cor}}_k   H^{2k+2}_{\mathrm{Dir}} \subset H^{2k}_{\mathrm{Dir}}.
$$
In other words, 
$$
\forall u\in H^{2k+2}_{\mathrm{Dir}} ,\forall \ell <k, \forall y\in \{0,1\}, \quad  \partial_{x}^{2\ell}  e^{- \mathfrak{C}_k } (\partial_x^2 + V) e^{\mathfrak{C}_k } u(y)=0
$$
or equivalently (note that $\partial_x^{2\ell} \partial_x^2 u(y)=0$),
$$
  e^{- \mathfrak{C}_k } (\partial_x^2 + V)  e^{ \mathfrak{C}_k } H^{2k+2}_{\mathrm{Dir}} \subset H^{2k}_{\mathrm{Dir}}  .
$$
Note that $(\partial_x^2 + V) H^{2k+2}_{V} \subset H^{2k}_{V} $, then, by means of Lemma \ref{lem:pas_si_mal}, we conclude the proof.
\end{proof}

Then, applying the Duhamel formula, we directly deduce the following corollary.
\begin{corollary}
\label{cor:cauchy}
For all $v^{(0)} \in H^{2k}_{\mathrm{Dir}}([0,1];\mathbb{C})$ there exists a unique solution 
$$
v\in C^0(\mathbb{R};H^{2k}_{\mathrm{Dir}}) \cap  C^1(\mathbb{R};H^{2k-2}_{\mathrm{Dir}}) 
$$
to the Cauchy problem
\begin{equation} \label{eq:cauchy_v}
 \ic \partial_t v = (\partial_x^2+V^{\mathrm{cor}}_k) v \quad \text{in}\; \R \times(0,1), \qquad v(0) = v^{(0)} \quad \text{in}\; \R.
\end{equation}
As usual, we denote it by $v(t) = e^{-\ic t (\partial_x^2+V^{\mathrm{cor}}_k)} v^{(0)}$.
\end{corollary}


\subsection{Dynamics} Finally, we show the equivalence between smooth solutions~\eqref{cond:reg} to problem~\eqref{eq:ev_Dir} and the existence of solutions to the above Cauchy problem~\eqref{eq:cauchy_v}.
\begin{proposition} \label{eq:prop_evident} Let $k\geq 1$.
A function
$$
u\in C^0(\mathbb{R};H^{2k}) \cap  C^1(\mathbb{R};H^{2k-2}) 
$$
is solution to \eqref{eq:ev_Dir} if and only if there exists $v^{(0)} \in H^{2k}_{\mathrm{Dir}}$ such that
\begin{equation}
\label{eq:leopart}
\forall t \in \mathbb{R}, \quad u(t) = e^{\mathfrak{C}_k} e^{-\ic t (\partial_x^2+V^{\mathrm{cor}}_k)} v^{(0)}.
\end{equation}
\end{proposition}

\begin{proof} Firstly, let $v^{(0)} \in H^{2k}_{\mathrm{Dir}}$ and $ u\in C^0(\mathbb{R};H^{2k}) \cap  C^1(\mathbb{R};H^{2k-2})$ satisfying \eqref{eq:leopart}. We shall show that $u$ is a solution to~\eqref{eq:ev_Dir}. Due to Theorem \ref{thm:main}, $ \mathfrak{C}_k  $ is bounded on $H^{2k-2}$  and $t\mapsto  e^{-\ic t (\partial_x^2+V^{\mathrm{cor}}_k)} v^{(0)}$ is $C^1$ with values in $H^{2k-2}$, thus we have
$$
\ic \partial_t u =  e^{ \mathfrak{C}_k  } (\partial_x^2+V^{\mathrm{cor}}_k) e^{-\ic t (\partial_x^2+V^{\mathrm{cor}}_k)} v^{(0)}.
$$
As a consequence, there holds
$$
\ic \partial_t   u = e^{ \mathfrak{C}_k   } (\partial_x^2+V^{\mathrm{cor}}_k) e^{-\ic t (\partial_x^2+V^{\mathrm{cor}}_k)} v^{(0)} =e^{ \mathfrak{C}_k   } (\partial_x^2+V^{\mathrm{cor}}_k) e^{ -\mathfrak{C}_k   }  u(t) = (\partial_x^2+V) u(t),
$$
where the last identity is just the definition of $V^{\mathrm{cor}}_k$. Finally, since $ e^{-\ic t (\partial_x^2+V^{\mathrm{cor}}_k)} v^{(0)}$ vanishes at the boundary (because it belongs to $H^{2k}_{\mathrm{Dir}}$) and $\exp( \mathfrak{C}_k) H^{2k}_{\mathrm{Dir}} \subset H^{2k}_{V} \subset  H^{2}_{\mathrm{Dir}}$ by Proposition~\ref{prop:constr_correct}, we deduce that $u$ also vanishes at the boundary. 

Now, we assume that $u$ is solution to \eqref{eq:ev_Dir}. We set $v = e^{- \mathfrak{C}_k }u$.
By composition, we directly have that $v\in C^0(\mathbb{R};H^{2k}) \cap  C^1(\mathbb{R};H^{2k-2}) $ and
$$
\ic \partial_t   v =   e^{- \mathfrak{C}_k } (\partial_x^2 + V ) u =   e^{- \mathfrak{C}_k } (\partial_x^2 + V ) e^{ \mathfrak{C}_k } v = (\partial_x^2 + V^{\mathrm{cor}}_k)v.
$$
Then, we set $v^{(0)}= v(0)$. By uniqueness of the solution to the Cauchy problem~\eqref{eq:cauchy_v}, to conclude the proof, it suffices  to prove that $v$ takes values in $ H^{2k}_{\mathrm{Dir}}$. To get this last property, by Proposition \ref{prop:compat_cond}, $u$ takes values in $H^{2k}_V$ and finally, $  e^{- \mathfrak{C}_k }$ maps  $H^{2k}_V$ in $H^{2k}_{\mathrm{Dir}}$ by Lemma \ref{lem:pas_si_mal}. 
\end{proof}


\subsection{Time discretization and splitting methods} \label{sub:time_disc} As a direct corollary of the results of this section, we deduce the convergence of a high order method for the time discretization. The method relies on the formula 
$$
u(t) = e^{\mathfrak{C}_k} e^{-\ic t (\partial_x^2+V^{\mathrm{cor}}_k)} e^{-\mathfrak{C}_k} u^{(0)}
$$
of Proposition \ref{eq:prop_evident}. First, we discretize  \eqref{eq:ev_Dir} in time by applying a splitting method. For a time step $\tau>0$, we approximate the solution $u$ at discrete time $t_n = n\tau$ by
\begin{equation}
\label{eq:def_time_disc}
u_n = e^{\mathfrak{C}_k} \big(e^{-\ic b_s \tau \partial^2_x} r(-\ic a_s \tau V^{\mathrm{cor}}_k) \cdots  e^{-\ic b_1 \tau \partial^2_x} r(-\ic a_1 \tau V^{\mathrm{cor}}_k)   \big)^n e^{-\mathfrak{C}_k} u^{(0)},
\end{equation}
where $a_1, b_1, \cdots,  a_s, b_s \in \mathbb{R}$ are the coefficients of a splitting method of formal order $p\geq 1$, $s \geq 1$ is the number of stages and $r$ is a real analytic function, defined on a neighbourhood of the origin, such that
$$
r(z) = e^z + \mathcal{O}(z^{p+1}), \qquad \text{as}\; z \to 0.
$$
Since we are not able to compute $ \exp(-\ic t V^{\mathrm{cor}}_k)$ exactly, we approximate $r$ by a Runge-Kutta method.
\begin{remark} The method in~\eqref{eq:def_time_disc} can be interpreted as a processed scheme. Thereby, for a constant time step $\tau$, a preferably computationally cheap integrator $K_{\tau}$ is enhanced with a change of variable $\pi_{\tau}$, the so-called \emph{post-processor},
$$
u_n = \pi_{\tau} \circ K^n_{\tau} \circ \pi^{-1}_{\tau}(u^{(0)}),
$$ 
in order to get the numerical solution $u_n$ of the original problem. This approach is favorable if the computation costs of the \emph{kernel} $K$ are cheap, since it is implemented $n$ times, while the processors $\pi_{\tau}, \pi^{-1}_{\tau}$ are computed only once. This idea was first introduced as concept of effective order~\cite{Butcher:1969:TEO}, to achieve an explicit Runge-Kutta method of order five with only five internal steps. There was a revived interest for this approach of effective order with the development of geometric numerical integration methods, where integration with a constant time step is natural~\cite{Blanes:1999:SIW, Hairer:2010:GNI, Leimkuhler:2004:SHD}. In general, processing is applied to increase the accuracy of a scheme while keeping the same (or a fewer) amount of evaluations per step. Reconsidering~\eqref{eq:def_time_disc}, the splitting scheme $\Phi^{\text{Y0cor}}_{\tau}$ corresponds to the kernel, and $e^{\mathfrak{C}_k}$ denotes the post-processor $\pi_{\tau}$ although it does not depend on $\tau$ (analogously, $e^{-\mathfrak{C}_k}$ is called the \emph{pre-processor} $\pi^{-1}_{\tau}$).
\end{remark}
The family of splitting methods given in~\eqref{eq:def_time_disc} preserves the regularity in space, i.e. the space $H^{2k}_{\mathrm{Dir}}$. As a corollary of Theorem \ref{thm:main}, applying the classical result for convergence of splitting methods \cite{Jahnke:2000:EBF}, we consider the following error estimate for this time discretization.
\begin{proposition}  \label{prop:disc_time}
If $1\leq a \leq k$ and $u^{(0)}\in H^{2k}_V$, then for all $T>0$, we have
$$
 \| u(t_n) - u_n  \|_{H^{2a}} \lesssim_{T,k} \tau^{\min(k-a,p)} \|u^{(0)} \|_{H^{2k}} \quad \text{for}\; 0 \leq t_n \leq T.
$$
\end{proposition}
\noindent Note that we use in general the following notations:
\begin{itemize}
\item $ \| a \|  \lesssim  \| b \|$ means bounded up to a constant $C$, $ \| a \|  \leq C \| b \|$, and
\item  $ \| a \|  \lesssim_{\beta_1, \ldots, \beta_m}  \| b \|$ means the constant may depend on $\beta_1, \ldots, \beta_m, m \geq 1$ (but not on $a$ and $b$), 
$$ \| a \|  \leq C(\beta_1, \ldots, \beta_m) \| b \|.$$
\end{itemize}

In Section~\ref{sec:convanal}, we prove Proposition~\ref{prop:disc_time} in a fully discretized setting.


\section{Space discretization} \label{sec:discret}
\sectionmark{\footnotesize Space discretization}

In remains to consider the space discretization of \eqref{eq:def_time_disc}. We aim at performing pseudo-spectral methods in order to get high accuracy. The main point is that we have to be cautious when discretizating $V^{\mathrm{cor}}_k$ because a bad discretization would easily lead to an unbounded operator (and so to instability). From now on, we consider $k\geq 2$, $u^{(0)} \in H^{2k}_V$ as fixed and the corrector function $\mathfrak{C}_k$ in terms of the associated coefficients $\alpha_{i,n,y}$ given by Proposition~\ref{prop:constr_correct}.

Throughout this section, we approximate the solution of \eqref{eq:ev_Dir} at time $t \in [0, T]$ by trigonometric polynomials, $u(t) \in \mathcal{P}_N$, where
$$
\mathcal{P}_N = \mathrm{Span}_{\mathbb{C}} (e^{-\pi \ic N \cdot},\cdots,e^{\pi \ic (N+1) \cdot}) \subset L^2(\mathbb{T}) ,\quad \mathrm{and} \quad  \mathbb{T} = \mathbb{R}/2\mathbb{Z}.
$$
and $N \gg 1$ is the parameter of discretization in space.


\subsection{Periodic extension} We have to represent numerically smooth functions which do not belong to $H^{2k}_{\mathrm{Dir}}$ because they do not satisfy the generalized Dirichlet boundary conditions. A natural approach is to consider these functions as restrictions on $[0,1]$ of smooth functions on $ \mathbb{T}$. In order to do so, we design an evolution equation whose solutions, when restricted to $(0,1)$, coincide with those of \eqref{eq:ev_Dir} (or its modified version \eqref{eq:cauchy_v}).

First, we introduce functions
\begin{itemize}
\item $W\in C^\infty(\mathbb{T};\mathbb{R})$ s.t. $W_{|[0,1]} = V$,
\item  $w^{(0)}\in H^{2k}(\mathbb{T})$ s.t. $w^{(0)}_{|[0,1]} = u^{(0)}$,
\item $\psi_{n,y}\in C^\infty(\mathbb{T};\mathbb{R})$ s.t. $\psi_{n,y} = \varphi_{n,y}$ on $(0,1)$ and $\psi_{n,y}=0$ in a neighbourhood of $-\frac{1}{2}$.
\end{itemize}
We consider the operator of extension by imparity, defined by
$$
\mathcal{E} : \left\{ \begin{array}{lll} \mathbb{C}^{[0,1]} &\to& \mathbb{C}^{\mathbb{T}} \\
u&\mapsto& \mathcal{E} u
\end{array}  \right. \qquad \mathrm{where} \qquad \left\{ \begin{array}{lll}  \mathcal{E} u(x) = u(x) \quad \mathrm{if} \quad x\in [0,1] \\ \mathcal{E} u(x) = -u(-x) \quad \mathrm{if} \quad x\in (-1,0)  \end{array}  \right.  
$$
We also introduce the operator of restriction 
$$
\mathcal{R} : \left\{ \begin{array}{lll} \mathbb{C}^{\mathbb{T}} &\to&  \mathbb{C}^{[0,1]} \\
u&\mapsto& u_{|[0,1]}
\end{array}  \right. 
$$
Furthermore, we denote by $\Lambda$ their composition $\Lambda = \mathcal{E}\mathcal{R}$.
We point out the following basic lemma that will be very useful.
\begin{lemma} \label{lem:bound_lambda_per} For all $j \geq 1$, $\Lambda$ is a bounded operator $ H^{2j}(\mathbb{T}) \cap \mathcal{R}^{-1} H^{2j}_{\mathrm{Dir}}([0,1]) \to H^{2j}(\mathbb{T})$.
\end{lemma}
Then, we extend the operator $\mathcal{I}_{n,y} u$ by setting
$$
\mathcal{J}_{n,y} u : \left\{ \begin{array}{lll} L^1(\mathbb{T}) &\to& L^1(\mathbb{T}) \\
u&\mapsto& \mathcal{J}_{n,y}u
\end{array}  \right. 
$$
where
$$
\forall x\in (-1/2,3/2], \quad \mathcal{J}_{n,y} u (x) =  \frac1{(n-1)!} \int_y^x (x-z)^{n-1}  u(z) \mathrm{d}z.
$$
Note that it is still an iterated integral, similarly to Definition~\ref{def:integral} (it is defined in this way by conciseness). Then, we define the analogue of $\mathfrak{C}_k$ on the torus by 
$$
\mathfrak{E}_k u = \sum_{y\in \{0,1\}} \sum_{i=1}^{k-1} \sum_{n=1}^{2i+1} \alpha_{i,n,y}  \psi_{2i+1-n,y} \mathcal{J}_{n,y} u.
$$
Since the functions $ \psi_{i,y}$ vanish in a neighbourhood of $\{-1/2\}$, the following result holds true by construction.
\begin{lemma} \label{lem:bound_E_per}For all $s\geq 0$, $\mathfrak{E}_k$ is bounded on $H^{s}(\mathbb{T})$.
\end{lemma}
Moreover, proceeding exactly as in the proof of Theorem \ref{thm:main}, we get the following lemma.
\begin{lemma} \label{lem:bound_bracket_per} For all $s>3/2$, $[\partial_x^2,\mathfrak{E}_k]$ is bounded on $H^{s}(\mathbb{T})$.
\end{lemma}
Furthermore, the spaces of odd periodic functions are denoted by
$$
H^s_{odd}(\mathbb{T}) = \{ u\in H^s(\mathbb{T}) \ | \ u = \Lambda u\}, \quad s\geq 0.
$$
\begin{remark}
Note that for all $j\geq 0$, $\mathcal{E}$ is bounded from $H^{2j}_{\mathrm{Dir}}([0,1]) $ to $H^{2j}_{odd}(\mathbb{T})$ and $\mathcal{R} $ is bounded form $ H^{2j}(\mathbb{T})$ to $ H^{2j}([0,1])$. Moreover, $\mathcal{R}$ maps $H^{2j}_{odd}(\mathbb{T})$ to $H^{2j}_{\mathrm{Dir}}([0,1])$.
\end{remark}
We can now define the modified potential by
$$
W^{\mathrm{cor}}_k =  \Lambda  e^{- \mathfrak{E}_k  } (\partial_x^2 + W) e^{ \mathfrak{E}_k  }- \partial_x^2,
$$
and we show the following result as a corollary of Theorem \ref{thm:main}.
\begin{corollary} \label{cor:bound_W}
For all $j\in \{1,\ldots,k\}$, $W^{\mathrm{cor}}_k$ is bounded on $H^{2j}_{odd}(\mathbb{T})$.
\end{corollary}
\begin{proof} 
First, we note that $\mathcal{R} \mathfrak{E}_k = \mathfrak{C}_k \mathcal{R} $ by construction. Thus, since $\mathfrak{E}_k $ is bounded on $L^2$, we have for all $t\in \mathbb{R}$,
$$
\forall u \in L^2(\mathbb{T}), \quad \mathcal{R} \exp(t \mathfrak{E}_k  ) u = \exp(t \mathfrak{E}_k  )\mathcal{R}u.
$$
Thus, since the operators $\partial_x^2$ and $W$ are local, we have
\begin{equation}
\label{eq:rutile}
\forall u \in H^2(\mathbb{T}), \quad  W^{\mathrm{cor}}_k u =\mathcal{E}V^{\mathrm{cor}}_k \mathcal{R} u  .
\end{equation}
Due to Theorem \ref{thm:main}, $V^{\mathrm{cor}}_k$ is bounded on $H^{2j}_{\mathrm{Dir}}$, and it follows that $W^{\mathrm{cor}}_k$ is bounded on $H^{2j}_{odd}(\mathbb{T})$.
\end{proof}

Finally, the following proposition states that we have designed a periodic extension of \eqref{eq:ev_Dir}.
\begin{proposition} \label{prop:period_ext}
 For all $t \in \mathbb{R}$, we have
 \begin{equation}
 \label{eq:relay}
 e^{-\ic t (\partial_x^2+V)} u^{(0)} = \mathcal{R}  e^{ \mathfrak{E}_k  } e^{-\ic t (\partial_x^2+W_k^{\mathrm{cor}})} \Lambda e^{ -\mathfrak{E}_k  } w^{(0)}.
 \end{equation}
\end{proposition}
\begin{proof}
First, as in the proof of Corollary \ref{cor:bound_W}, we have $\mathcal{R}  e^{ t\mathfrak{E}_k  } =  e^{ t\mathfrak{E}_k  } \mathcal{R}$. Thus, we have by construction
$$
\mathcal{R} \Lambda e^{ -\mathfrak{E}_k  } w^{(0)} = e^{ -\mathfrak{C}_k  } u^{(0)}.
$$
By Theorem \ref{thm:main}, we have $\mathcal{R}e^{ -\mathfrak{E}_k  } w^{(0)} \in H^{2k}_{\mathrm{Dir}}$, and thus, $\Lambda e^{ -\mathfrak{E}_k  } w^{(0)} \in H^{2k}_{odd}$. Then we set
$$
\forall t \in \mathbb{R}, \quad z(t) = e^{-\ic t (\partial_x^2+W_k^{\mathrm{cor}})} \Lambda e^{ -\mathfrak{E}_k  } w^{(0)}.
$$
and we have $z(t) \in H^{2k}_{odd} $ for all $t\in \mathbb{R}$.  By construction and using \eqref{eq:rutile}, its restriction satisfies
$$
\ic \partial_t \mathcal{R} z =  \mathcal{R} (\partial_x^2+W_k^{\mathrm{cor}}) z= (\partial_x^2+V_k^{\mathrm{cor}})\mathcal{R}  z.
$$
Since $ \mathcal{R} z  \in C^0( \mathbb{R};H^{2k}_{\mathrm{Dir}})$ it follows by uniqueness of the solution to \eqref{eq:cauchy_v} that
$$
\forall t \in \mathbb{R}, \quad \mathcal{R} z(t) = e^{-\ic t (\partial_x^2+V_k^{\mathrm{cor}})}\mathcal{R} z(0) = e^{-\ic t (\partial_x^2+V_k^{\mathrm{cor}})}e^{ -\mathfrak{C}_k  } u^{(0)}.
$$
Finally, we conclude by Proposition \ref{eq:prop_evident} that \eqref{eq:relay} holds.
\end{proof}


\subsection{Discretization in space by pseudo-spectral methods} It remains to perform the space discretization. We have to discretize $\partial_x^2$, $\mathfrak{E}_k, W^{\mathrm{cor}}_k$ as operators acting on $\mathcal{P}_N$. First, we define the discrete torus
$$
\mathbb{T}_N = \Delta x \big( \mathbb{Z}/(2N+2) \mathbb{Z} \big),
$$
for a grid spacing $\Delta x = \frac{1}{N+1}$.
Then, we introduce the discrete Fourier transform by setting
$$
\mathcal{F}_N^{-1} : \left\{ \begin{array}{lll} \mathcal{P}_N &\to &\mathbb{C}^{\mathbb{T}_N } \\ w & \mapsto & w_{| \mathbb{T}_N} \end{array} \right.
$$
Note that the usual formula holds too. Furthermore, we denote by $\Pi_{\mathcal{P}_N}$ the orthogonal projector on $\mathcal{P}_N$, i.e. $\Pi_{\mathcal{P}_N} : L^2(\mathbb{T}) \to \mathcal{P}_N$,
$$
\Pi_{\mathcal{P}_N} \sum_{k\in \mathbb{Z}} w_k e^{\ic \pi k \cdot} =  \sum_{k=-N}^{N+1} w_k e^{\ic \pi k \cdot},    \quad \forall w \in \mathbb{C}^{\mathbb{Z}}.
$$
We introduce the operator of restriction to the grid
$$
\mathcal{A}_N :\left\{ \begin{array}{lll} C^0(\mathbb{T}) &\to & \mathcal{P}_N  \\ w & \mapsto & \mathcal{F}_N w_{| \mathbb{T}_N} \end{array} \right.
$$
Moreover, we introduce the desaliased product
\begin{equation*} 
\forall v,w \in \mathcal{P}_N, \quad v \diamond w = \Pi_{\mathcal{P}_N} (vw).
\end{equation*}
This desaliased product can easily be computed in practice thanks to the following classical lemma.
\begin{lemma}
For all $v,w \in \mathcal{P}_N$, we have $v \diamond w = \Pi_{\mathcal{P}_N} \mathcal{A}_{2N+1}(vw)$.
\end{lemma}
\begin{proof}
We note that $vw \in  \mathcal{P}_{2N+1}$. As a consequence, we have that $vw = \mathcal{A}_{2N+1}(vw)$ and so $v \diamond w = \Pi_{\mathcal{P}_N} \mathcal{A}_{2N+1}(vw)$.
\end{proof}
Note that the coefficient $2$ could be replaced by $3/2$ but it would require extra notations. Moreover note that $\Delta x$ becomes $(4N+4)^{-1}$ when considering $\mathcal{A}_{2N+1}$.

In order to define the operator, we discretize the functions $\psi_{2i+1-n,y}$ and $w^{(0)}$ by setting
 $$
 \psi_{2i+1-n,y,N} = \mathcal{A}_N  \psi_{2i+1-n,y}, \quad  \mathrm{and} \quad w^{(0)}_N =  \mathcal{A}_N  w^{(0)}.
 $$
Then, we discretize $\mathfrak{E}_k$ by setting 
\begin{equation*} 
\begin{split}
\mathfrak{E}_{k,N} =& \sum_{y\in \{0,1\}} \sum_{i=1}^{k-1} \sum_{n=2}^{2i+1} \alpha_{i,n,y} \mathcal{A}_N  ( \psi_{2i+1-n,y} \mathcal{J}_{n,y} )  \\
&+ \sum_{y\in \{0,1\}} \sum_{i=1}^{k-1} \alpha_{i,1,y}  \big( \psi_{2i,y,N} \diamond  \mathcal{J}_{1,y}^{\mathrm{main}} + \mathcal{A}_N (  \psi_{2i,y} \mathcal{J}_{1,y}^{\mathrm{bound}}  ) \big)
\end{split}
\end{equation*}
where we use the notations
$$
\mathcal{J}_{n,y}^{\mathrm{main}}  e^{\ic \pi k \cdot} = \mathds{1}_{k\neq 0} (\ic \pi k)^{-n}  e^{\ic \pi k \cdot} \quad \text{and} \quad \mathcal{J}_{n,y}^{\mathrm{bound}} = \mathcal{J}_{n,y} - \mathcal{J}_{n,y}^{\mathrm{main}} .
$$
Note that for all $u\in L^1(\mathbb{T})$ and all $x\in (-1/2,3/2]$,
\begin{equation*}
\mathcal{J}_{n,y}^{\mathrm{bound}} u (x) = \frac{(x-y)^n}{n!} \int_{\mathbb{T}} u(z) \frac{\mathrm{d}z}2 - \sum_{j=0}^{n-1} \frac{(x-y)^j}{j!} \mathcal{J}_{n-j,y}^{\mathrm{main}}u(y).
\end{equation*}
 Then  we discretize $\Lambda$ by setting
$$
\Lambda_N w = \mathcal{F}_N v,
$$
where $v\in \mathbb{C}^{\mathbb{T}_N}$ is defined by
$$
v(x) = w(x), \quad \forall x\in \Delta x\{1,\ldots,N\} \quad \mathrm{and} \quad v(-x) = -v(x) \quad \forall x \in \mathbb{T}_N.
$$ 
We discretize the modified, extended potential $W^{\mathrm{cor}}_k$ by setting
\begin{equation*}
W^{\mathrm{cor}}_{k,N} =  \Lambda_N  e^{- \mathfrak{E}_{k,N}  } (\partial_x^2 + \mathcal{A}_N (W\cdot) ) e^{ \mathfrak{E}_{k,N}  }- \partial_x^2 .
\end{equation*}

Finally, we state the following convergence result for this space discretization.
\begin{proposition} \label{prop:cvg_space}
Let $T>0$ and $t\in [0,T]$, then we denote
\begin{equation} \label{def:v_vN}
v(t) = e^{ \mathfrak{E}_k  } e^{-\ic t (\partial_x^2+W_k^{\mathrm{cor}})} \Lambda e^{ -\mathfrak{E}_k  } w^{(0)}, \quad \mathrm{and} \quad v_N(t) =  e^{ \mathfrak{E}_{k,N}  } e^{-\ic t (\partial_x^2+W_{k,N}^{\mathrm{cor}})} \Lambda_N e^{ -\mathfrak{E}_{k,N}  } w_N^{(0)}.
\end{equation}
If $1\leq a \leq k$, there holds
\begin{equation} \label{eq:lapin}
 \| v(t) - v_N(t)  \|_{H^{2a}} \! \! \lesssim_{T} \! \!  N^{-2(k-a)} \|w^{(0)} \|_{H^{2k}}.
\end{equation}
\end{proposition}
Note that due to Proposition \ref{prop:period_ext}, $v(t)$ is a periodic extension to $\exp(-\ic t (\partial_x^2+V))u^{(0)}$.
The proof of this proposition requires many technical lemmas, thus we postpone it to Section \ref{sec:errorest}. 


\section{Convergence analysis in a fully discretized setting} \label{sec:convanal}
\sectionmark{\footnotesize Convergence analysis in a fully discretized setting}
\subsection{Consistency and stability} In this subsection, we prove consistency and stability estimates for the discretized operators. Motivated by Lemma \ref{lem:consis_lambda} below, we introduce the following ad-hoc norm
$$
\forall u \in H^2, \quad \| u \|_{H^2_N} =  \| u\|_{H^2} + N^{3/2} (|u(0)| + |u(1)|).
$$
In a first part, we state and prove boundedness for some operators we introduced in the previous section. We recall the following standard result, which is a consequence of the Poisson formula, see e.g. \cite[Lemma 3.9]{Abou:2025:ACO}.
\begin{lemma} \label{lem:bound_A} For all $s>1/2$, all $N\geq 1$ and all $f \in H^s$, we have $\| \mathcal{A}_N f \|_{H^s} \lesssim_s \| f\|_{H^s}$.
\end{lemma}

As a corollary, we deduce the following basic lemma.
\begin{lemma}
\label{lem:consis_AN}
For all $s_1\geq s_2>1/2$, for all $u\in H^{s_1}$ and all $N\geq 1$, we have
$$
\| u - \mathcal{A}_N u\|_{H^{s_2}} \lesssim_{s_2} N^{-(s_1-s_2)} \| u\|_{H^{s_1}}.
$$
\end{lemma}
\begin{proof} First, we note that $ \mathcal{A}_N \Pi_{\mathcal{P}_N} u = \Pi_{\mathcal{P}_N} u$.
Thus, by Lemma \ref{lem:bound_A}, we obtain
\begin{equation*}
\begin{split}
\| u - \mathcal{A}_N u\|_{H^{s_2}} &\leq \| u - \Pi_{\mathcal{P}_N} u \|_{H^{s_2}} + \| \mathcal{A}_N(u - \Pi_{\mathcal{P}_N} u) \|_{H^{s_2}} \\ &\lesssim_{s_2} \| u - \Pi_{\mathcal{P}_N} u \|_{H^{s_2}} \lesssim_{s_2} N^{-(s_1-s_2)}  \| u - \Pi_{\mathcal{P}_N} u \|_{H^{s_1}}  \lesssim_{s_2} N^{-(s_1-s_2)} \| u\|_{H^{s_1}},
\end{split}
\end{equation*}
which concludes the proof.
\end{proof}

We also recall the Plancherel identity in this context.
\begin{lemma} \label{lem:Planch} For all $u\in \mathcal{P}_N$ and all $N\geq 1$, we have
$$
\int_{\mathbb{T}} |u(x)|^2 \frac{\mathrm{d}x}2 = \| u\|_{L^2}^2 = \frac{\Delta x}2 \sum_{x\in \mathbb{T}_N} |u(x)|^2.
$$
\end{lemma}

Then, we study the operator $\Lambda_N$.
\begin{lemma} \label{lem:alg_Lambda} For all $N\geq 1$ and all $u\in \mathcal{P}_N$, $\Lambda_N u \in H^\infty_{odd}(\mathbb{T})$ is an odd function. Moreover, if $u(0)=u(1)=0$, we have  $\Lambda_N u = \mathcal{A}_N \Lambda u$.
\end{lemma}
\begin{proof} Let $f = \Lambda_N u$ and let
$$
f = \sum_{k=-N}^{N+1} f_k e^{\ic\pi  k \cdot}
$$
be its Fourier decomposition. Since
$$
f_k = \frac{\Delta x}2 \sum_{x\in \mathbb{T}_N} f(x) e^{-\ic \pi k x}
$$
and $f(x)=-f(-x)$ for all $x\in \mathbb{T}_N$ by definition of $\Lambda$, there holds that
$$
f =  \sum_{k=1}^{N} 2\ic f_k \sin(k\pi \cdot),
$$
and therefore, $f$ is an odd function on $\mathbb{T}$.

Moreover, we note that if $u$ vanishes at the boundary, $ \Lambda u \in H^2$, and so it makes sense to consider $\mathcal{A}_N \Lambda u$. Thus, it suffices to note that $\mathcal{A}_N \Lambda u \in \mathcal{P}_N$ is an odd function on $\mathbb{T}_N$, which is equal to $u$ on $\Delta x \{1,\ldots,N\}$.
\end{proof}

\begin{lemma} \label{lem:consis_lambda} Let $N\geq 1$, $v\in H^{2k}(\mathbb{T})$ and $v_N \in \mathcal{P}_N$ be such that $\mathcal{R} v\in H^{2k}_{\mathrm{Dir}}$. Then, we have
$$
\| \Lambda v - \Lambda_N v_N \|_{H^2} \lesssim \| v-v_N \|_{H^2_N} + N^{-2k+2} \| v\|_{H^{2k}}.
$$
and, a fortiori,
$$
\| \Lambda_N v_N \|_{H^2} \lesssim \| v_N\|_{H^2_N}.
$$
\end{lemma}
\begin{proof} Let $\Psi_{0,N},\Psi_{1,N} \in \mathcal{P}_N$ be such that
$$
\forall y\in \{0,1\}, \forall x\in \mathbb{T}_N, \quad \Psi_{y,N}(x) = \mathds{1}_{x=y}.
$$
We note that by Lemma~\ref{lem:Planch}, there holds
$$
\|  \Psi_{0,N} \|_{L^2} = \|  \Psi_{1,N} \|_{L^2} = \sqrt{\frac{\Delta x}{2}}.
$$
Thus, for all $y\in \{0,1\}$, we obtain
$$
\|  \Psi_{y,N} \|_{H^2} \leq N^2 \|  \Psi_{y,N} \|_{L^2} \lesssim N^2\sqrt{ \Delta x} \sim N^{3/2}.
$$
Then, we set 
$$
w_N = v_N  - v_N(0) \Psi_{0,N} - v_N(1) \Psi_{1,N}.
$$
By construction and Lemma \ref{lem:alg_Lambda}, we have
$$
\Lambda_N v_N - \Lambda v = \Lambda_N w_N - \Lambda v = \mathcal{A}_N\Lambda (w_N-v) + \mathcal{A}_N \Lambda v- \Lambda v.
$$ 
On the one hand, since $\mathcal{R} v\in H^{2k}_{\mathrm{Dir}}$, we obtain by Lemma \ref{lem:consis_AN},
$$
\| \mathcal{A}_N \Lambda v- \Lambda v \|_{H^2} \lesssim N^{-2k+2} \| \Lambda v \|_{H^{2k}} \lesssim N^{-2k+2} \| v \|_{H^{2k}}.
$$
On the other hand, by Lemma \ref{lem:bound_A}, since $w_N(0) = w_N(1)=0$, we deduce
$$
\| \mathcal{A}_N\Lambda (w_N-v)  \|_{H^2} \lesssim \| \Lambda (w_N-v)  \|_{H^2} \lesssim \| w_N-v  \|_{H^2}.
$$
Since $\| \Psi_{0,N} \|_{H^2} + \| \Psi_{1,N} \|_{H^2} \lesssim N^{3/2}$ and $v$ vanishes at the boundary, we conclude that
$$
 \| w_N-v  \|_{H^2} \lesssim  \| v-v_N \|_{H^2_N}
$$
holds true.
\end{proof}

Furthermore, we prove a continuity estimate for the desaliased product.
\begin{lemma} \label{lem:diam} For all $N\geq 1$, $s>1/2$, $u,v \in \mathcal{P}_N$ and $y\in \{0,1\}$, we have
$$
\| u \diamond v\|_{H^s} \lesssim_s \| u\|_{H^s} \| v \|_{H^s} \quad \mathrm{and} \quad  |(u \diamond v)(y) | \lesssim N^{-3/2} \| u\|_{H^2} (\| v\|_{H^2} + N^{3/2} |v(y)|),
$$
and
\begin{equation} \label{eq:uvdes}
 \| u \diamond v\|_{H^2_N} \lesssim  \| u\|_{H^2} \| v \|_{H^2_N}.
\end{equation}
\end{lemma}
\begin{proof}
Since $\Pi_{\mathcal{P}_N}$ is an orthogonal projection, the $H^s$-bound follows directly,
$$
\| u \diamond v\|_{H^s} =  \|  \Pi_{\mathcal{P}_N} (u v)\|_{H^s} \leq \| u v \|_{H^s}   \lesssim \| u\|_{H^s} \| v \|_{H^s} .
$$
Then, we prove the bound for $(u \diamond v)(0)$ (the one in $y=1$ is similar). Setting $w = v-v(0)$,
we have that $uw \in \mathcal{P}_{2N+1}$. We use the $H^2$-bound to show the following estimate,
\begin{align*}
|(u \diamond v)(0) | \leq  |v(0)| |(u \diamond 1)(0) | + | (u \diamond w) (0)|  \lesssim  |v(0)| \| u\|_{H^2} + |(uw)(0) - (u \diamond w) (0)|. 
\end{align*}
Furthermore, by Lemma \ref{lem:Planch} we obtain that
\begin{align*}
|(uw)(0) - (u \diamond w) (0)| \lesssim  N^{1/2} \| uw - u \diamond w \|_{L^2} = N^{1/2} \| (\mathrm{Id} - \Pi_{\mathcal{P}_N})uw\|_{L^2} \leq N^{-3/2} \| uw\|_{H^2}.
\end{align*}
Finally, we deduce the desired estimate,
\begin{align*}
|(u \diamond v)(0) |  \lesssim N^{-3/2} \| u\|_{H^2} (\| v\|_{H^2} + N^{3/2} |v(0)|).
\end{align*}
Note that~\eqref{eq:uvdes} is then an immediate consequence of the definition of the norm $\norm{\, \cdot\,}_{H^2_N}$.
\end{proof}

In the following, we state and prove consistency and stability for the discretized correctors.

\begin{lemma} \label{lem:consis_E} For all $N\geq 1$, $u\in \mathcal{P}_N$, we have
$$
\| \mathfrak{E}_{k} u - \mathfrak{E}_{k,N} u \|_{H^2_N} \lesssim N^{-2k+2} \|u\|_{H^{2k}}.
$$
\end{lemma}
\begin{proof} We have to expand $\mathfrak{E}_{k}$ and $\mathfrak{E}_{k,N}$. So let  $y\in \{0,1\}$, $1\leq i\leq k-1$ and $1\leq n\leq 2i+1$. First, we note that if $n>1$, then $ \psi_{2i+1-n,y} \mathcal{J}_{n,y} $ vanishes at the boundary. Therefore, by Lemma~\ref{lem:consis_AN}, there holds
\begin{multline*}
\| (\psi_{2i+1-n,y} \mathcal{J}_{n,y})u -  \mathcal{A}_N(\psi_{2i+1-n,y} \mathcal{J}_{n,y})u \|_{H^2_N} = \| (\psi_{2i+1-n,y} \mathcal{J}_{n,y})u -  \mathcal{A}_N(\psi_{2i+1-n,y} \mathcal{J}_{n,y})u \|_{H^2} \\ \lesssim N^{-2k+2} \| \psi_{2i+1-n,y} \mathcal{J}_{n,y}u \|_{H^{2k}} \lesssim N^{-2k+2}\| u\|_{H^{2k}}.
\end{multline*}
Then we focus on the case $n=1$, which is more delicate. First, since $\psi_{2i,y} \mathcal{J}_{1,y}^{\mathrm{bound}} u$ vanishes at the boundary, we have again by Lemma~\ref{lem:consis_AN},
$$
\| (\psi_{2i,y} \mathcal{J}_{ 1,y}^{\mathrm{bound}} u) -  \mathcal{A}_N(\psi_{2i,y} \mathcal{J}_{1,y}^{\mathrm{bound}}u) \|_{H^2_N} =\| (\mathrm{Id} - \mathcal{A}_N) (\psi_{2i,y} \mathcal{J}_{ 1,y}^{\mathrm{bound}} u)  \|_{H^2}   \lesssim N^{-2k+2}\| u\|_{H^{2k}}.
$$
Similarly, we get
$$
\| \psi_{2i,y,N}  \mathcal{J}_{1,y}^{\mathrm{main}}u  -  \psi_{2i,y}  \mathcal{J}_{1,y}^{\mathrm{main}}u \|_{H^2_N} = \| \psi_{2i,y,N}  \mathcal{J}_{1,y}^{\mathrm{main}}u  -  \psi_{2i,y}  \mathcal{J}_{1,y}^{\mathrm{main}}u \|_{H^2}   \leq  N^{-2k+2}\| u\|_{H^{2k}}.
$$
Thus, it only remains to control
$$
f=\psi_{2i,y,N} \diamond  \mathcal{J}_{1,y}^{\mathrm{main}}u  - \psi_{2i,y,N}  \mathcal{J}_{1,y}^{\mathrm{main}}u .
$$
Observing that $f\in \mathcal{P}_{2N+1}$ is a trigonometric polynomial of degree smaller or equal to $2N+2$, we deduce by Plancherel that 
$$
\| f \|_{H^2_N}  \lesssim \| f\|_{H^2} + N^2 \| f\|_{L^2} \lesssim  N^2  \| f\|_{L^2}.
$$
Finally, by definition of the desaliased product, we have 
\begin{equation*}
\begin{split}
\| f\|_{L^2} \leq&  \| (\mathrm{Id}-\Pi_{\mathcal{P}_N}) (\psi_{2i,y,N} \mathcal{J}_{1,y}^{\mathrm{main}}u)  \|_{L^2}  \lesssim N^{-2k}\| u\|_{H^{2k}},
\end{split}
\end{equation*}
which concludes the proof.
\end{proof}
Using that $\mathfrak{E}_{k}$ is bounded on $H^2$ and vanishes at the boundary, we deduce the following result as a corollary.
\begin{corollary} \label{lem:bound_E} For all $N\geq 1$ and all $u\in \mathcal{P}_N$, there holds $\|\mathfrak{E}_{k,N} u \|_{H^2_N} \lesssim  \| u\|_{H^2}$.
\end{corollary}
Finally, we prove the following consistency result for the corrector and its flow.
\begin{lemma} \label{lem:consis_exp_E}For all $N\geq 1$,  $v\in H^{2k}$, $v_N \in \mathcal{P}_N$,  $t\in [-1,1]$, we have
\begin{equation}
\label{eq:banane}
\| \mathfrak{E}_{k} v - \mathfrak{E}_{k,N} v_N \|_{H^2_N} \lesssim \| v-v_N \|_{H^2} + N^{-2k+2} \| v\|_{H^{2k}}.
\end{equation}
and 
\begin{equation}
\label{eq:et_ouai}
\| e^{t\mathfrak{E}_{k}} v - e^{t\mathfrak{E}_{k,N}} v_N \|_{H^2_N} \lesssim \| v-v_N \|_{H^2_N} + N^{-2k+2} \| v\|_{H^{2k}}.
\end{equation}
\end{lemma}
\begin{proof} First, let us prove that \eqref{eq:banane} implies \eqref{eq:et_ouai}. We set $w(t) = e^{t\mathfrak{E}_{k}} v $ and $w_N(t) = e^{t\mathfrak{E}_{k,N}} v_N$. We note by \eqref{eq:banane} that
\begin{equation*}
\begin{split}
\| \partial_t (w - w_N) \|_{H^2_N} = \| \mathfrak{E}_{k} w - \mathfrak{E}_{k,N} w_N \|_{H^2_N} \lesssim \|  w - w_N \|_{H^2_N} + N^{-2k+2} \sup_{|\tau|\leq 1}\| w(\tau)\|_{H^{2k}} .
\end{split}
\end{equation*}
Then, since $\mathfrak{E}_{k}$ is bounded on $H^{2k}$, see Lemma \ref{lem:bound_E_per}, we deduce
$$
\sup_{|\tau|\leq 1}\| w(\tau)\|_{H^{2k}} \lesssim \| v\|_{H^{2k}}.
$$
Finally, to get \eqref{eq:et_ouai}, it suffices to apply the Gr\"onwall inequality.

Now, we focus on \eqref{eq:banane}. We set $u_N = \mathcal{A}_N v$ and, since $\mathfrak{E}_{k} (v-u_N)$ vanishes at the boundary, we have
$$
\| \mathfrak{E}_{k} v - \mathfrak{E}_{k,N} v_N \|_{H^2_N} \leq \| \mathfrak{E}_{k} (v-u_N) \|_{H^2} + \| ( \mathfrak{E}_{k} - \mathfrak{E}_{k,N}) u_N \|_{H^2_N} + \| \mathfrak{E}_{k,N} (u_N-v_N) \|_{H^2_N}.
$$
We control the first term by applying Lemmata \ref{lem:consis_AN} and \ref{lem:bound_E_per}. We get
$$
\| \mathfrak{E}_{k} (v-u_N) \|_{H^2} \lesssim \| v-u_N \|_{H^2} \lesssim N^{-2k+2} \| v\|_{H^{2k}}.
$$
For the second term, we apply Lemma \ref{lem:consis_E} and we obtain by Lemma \ref{lem:bound_A},
$$
\| ( \mathfrak{E}_{k} - \mathfrak{E}_{k,N}) u_N \|_{H^2_N}  \lesssim N^{-2k+2} \| u_N \|_{H^{2k}} \lesssim N^{-2k+2} \| v\|_{H^{2k}}.
$$
Finally, for the last term, we apply Corollary \ref{lem:bound_E} to get
$$
\| \mathfrak{E}_{k,N} (u_N-v_N) \|_{H^2_N} \lesssim \| u_N-v_N \|_{H^2} \lesssim \| v-v_N \|_{H^2} + \| u_N-v \|_{H^2} \lesssim \| v-v_N \|_{H^2} + N^{-2k+2}\| v \|_{H^{2k}},
$$
which is the desired estimate~\eqref{eq:banane}.
\end{proof}

Now, we study the consistency and stability of the discretization of $W_k^{\mathrm{cor}}$. The proof will mainly rely on the following critical lemma.

\begin{lemma} \label{lem:consis_crochet}
For all $N\geq 1$ and all $u\in \mathcal{P}_N$, we have
\begin{equation}
\label{eq:vrai_truc}
\| [\partial_x^2,\mathfrak{E}_{k,N}]u - [\partial_x^2,\mathfrak{E}_{k}]u \|_{H^2_N} \lesssim  N^{-2k+2}\| u \|_{H^{2k}}
\end{equation}
and, a fortiori,
\begin{equation}
\label{eq:pas_grand_chose}
 \| [\partial_x^2,\mathfrak{E}_{k,N}]u \|_{H^2_N} \lesssim \| u\|_{H^2}  .
\end{equation}
\end{lemma}
\begin{proof} First, we note that to get \eqref{eq:pas_grand_chose} from \eqref{eq:vrai_truc}, it suffices to use that $ [\partial_x^2,\mathfrak{E}_{k}]u$ vanishes at the boundary and that $[\partial_x^2,\mathfrak{E}_{k}]$ is bounded on $H^2$ (Lemma \ref{lem:bound_bracket_per}). Then we are going to expand $\mathfrak{E}_{k,N}- \mathfrak{E}_{k}$ and to analyze the different terms arising.
So, let $y\in \{0,1\}$, $1\leq i\leq k-1$ and $1\leq n\leq 2i+1$.

In a first step, we estimate the term 
$$
E_1=[\partial_x^2, \psi_{2i,y,N} \diamond  \mathcal{J}_{1,y}^{\mathrm{main}} + \mathcal{A}_N (  \psi_{2i,y} \mathcal{J}_{1,y}^{\mathrm{bound}}  ) - \psi_{2i,y} \mathcal{J}_{1,y}  ]w_N.
$$
As in the continuous case, $E_1$ is critical. It requires the desalised product. Indeed, $ \partial_x^2\mathcal{J}_{1,y}^{\mathrm{main}} = \partial_x$ and $\partial_x (w\diamond v) =(\partial_x w)\diamond v +w\diamond (\partial_x  v ) $ for all $v,w\in \mathcal{P}_N$, and thus, we have
\begin{equation*}
\begin{split}
[ \partial_x^2,\psi_{2i,y,N} \diamond  \mathcal{J}_{1,y}^{\mathrm{main}}]u=   2 \partial_x \psi_{2i,y,N} \diamond  u+ \partial_x^2 \psi_{2i,y,N} \diamond  \mathcal{J}_{1,y}^{\mathrm{main}}u.
\end{split}
\end{equation*}
Then, we use that  
$$
\mathcal{J}_{1,y}^{\mathrm{bound}} u (x) = (x-y) \int_{\mathbb{T}} u(z) \frac{\mathrm{d}z}2 -  \mathcal{J}_{1,y}^{\mathrm{main}}u(y),
$$
to deduce that
$$
[\partial_x^2, \mathcal{A}_N (  \psi_{2i,y} \mathcal{J}_{1,y}^{\mathrm{bound}}  ) ]u = \partial_x^2 \mathcal{A}_N ((x-y)  \psi_{2i,y}  )\int_{\mathbb{T}} u(z) \frac{\mathrm{d}z}2 - \partial_x^2 \psi_{2i,y,N}  \mathcal{J}_{1,y}^{\mathrm{main}}u(y) + \psi_{2i,y,N} \partial_x u(y).
$$
Using that a similar decomposition holds for $[ \partial_x^2,\psi_{2i,y} \mathcal{J}_{1,y}]$ (see the proof of Theorem \ref{thm:main}), we deduce that
$$
\| E_1 \|_{H^2_N}\lesssim   E_{11} + E_{12} + E_{13} + E_{14}, 
$$
where
\begin{equation*}
\begin{split}
E_{11} &=  \| (\partial_x \psi_{2i,y,N}) \diamond u -  \partial_x \psi_{2i,y} u \|_{H^2_N},\\
E_{12} &= \|  \partial_x^2 \mathcal{A}_N ((x-y)  \psi_{2i,y}  ) - \partial_x^2 ( (x-y)  \psi_{2i,y} )  \|_{H^2_N} \| u \|_{L^1},\\
E_{13} &=  \| \psi_{2i,y,N}-\psi_{2i,y} \|_{H^2_N} |\partial_x u(y)|, \\
 E_{14} &= \| (\partial_x^2 \psi_{2i,y,N}) \diamond  v  -   (\partial_x^2 \psi_{2i,y})   v  \|_{H^2_N}
\end{split}
\end{equation*}
and
$$
v = \mathcal{J}_{1,y}^{\mathrm{main}} u- \mathcal{J}_{1,y}^{\mathrm{main}}u(y).
$$
First, by Lemmata \ref{lem:bound_A},  \ref{lem:Planch} and \ref{lem:consis_AN}, since $\partial_x \psi_{2i,y,N}  u \in \mathcal{P}_{2N+1}$, we have
\begin{equation*}
\begin{split}
E_{11} \leq & \| (\mathrm{Id}-\Pi_{\mathcal{P}_N}) (\partial_x \psi_{2i,y,N}  u) \|_{H^2_N} + \| \partial_x (\psi_{2i,y} - \psi_{2i,y,N} )  u \|_{H^2_N}\\
\lesssim &\| (\mathrm{Id}-\Pi_{\mathcal{P}_N}) (\partial_x \psi_{2i,y,N}  u) \|_{H^2} + N^2  \| (\mathrm{Id}-\Pi_{\mathcal{P}_N}) (\partial_x \psi_{2i,y,N}  u) \|_{L^2} \\
& +N^{3/2} \| \partial_x (\psi_{2i,y} - \psi_{2i,y,N} )  u \|_{H^2} \\
\lesssim & N^{-(2k-2)}\| \partial_x \psi_{2i,y,N}  u \|_{H^{2k}} +  N^{3/2} \| \psi_{2i,y} - \psi_{2i,y,N}\|_{H^3} \|u \|_{H^2} \\
\lesssim & N^{-(2k-2)} ( \| \partial_x \psi_{2i,y,N} \|_{H^{2k}} \|u \|_{H^{2k}} + \| \psi_{2i,y} - \psi_{2i,y,N}\|_{H^{2k + 6}} \|u \|_{H^2} ) \lesssim N^{-(2k-2)}  \|u \|_{H^{2k}} .
\end{split}
\end{equation*}
Then, the estimates on $E_{12}$ and $E_{13}$ present no difficulty and, due to Lemma~\ref{lem:consis_AN}, it follows directly
$$
E_{12} + E_{13} \lesssim N^{-2k+2} \|u\|_{H^2}.
$$
Finally, since $(\partial_x^2 \psi_{2i,y,N})  v  \in \mathcal{P}_{2N+1}$, still using the same lemmas, we have
\begin{equation*}
\begin{split}
E_{14} \leq & \|  (\mathrm{Id}-\Pi_{\mathcal{P}_N}) (\partial_x^2 \psi_{2i,y,N})  v \|_{H^2_N} +\| (\partial_x^2 \psi_{2i,y,N} -\partial_x^2 \psi_{2i,y} )  v \|_{H^2_N} \\
\lesssim & \|  (\mathrm{Id}-\Pi_{\mathcal{P}_N}) (\partial_x^2 \psi_{2i,y,N})  v \|_{H^2} + N^2  \|  (\mathrm{Id}-\Pi_{\mathcal{P}_N}) (\partial_x^2 \psi_{2i,y,N})  v \|_{L^2}  \\
&+ N^{3/2}\| \psi_{2i,y,N} -\psi_{2i,y}  \|_{H^4}  \| v \|_{H^2} \\
\lesssim & N^{-2k+2}  \| (\partial_x^2 \psi_{2i,y,N})  v  \|_{H^{2k}} + N^{-2k+2} \|v \|_{H^2} \lesssim  N^{-2k+2} \|v \|_{H^{2k}}.
\end{split}
\end{equation*}

In a second step, we estimate the term 
$[\partial_x^2,\mathcal{A}_N  ( \psi_{2i+1-n,y} \mathcal{J}_{n,y} ) - \psi_{2i+1-n,y} \mathcal{J}_{n,y} ]u$ with $n\geq 2$.

On the one hand, since $ \psi_{2i+1-n,y} \mathcal{J}_{n,y} \partial_x^2  u$ vanishes at the boundary, we have by Lemma~\ref{lem:consis_AN} 
\begin{equation*}
\begin{split}
&\| \mathcal{A}_N  ( \psi_{2i+1-n,y} \mathcal{J}_{n,y} \partial_x^2  u ) - \psi_{2i+1-n,y} \mathcal{J}_{n,y} \partial_x^2  u \|_{H^2_N} \\
& =\| \mathcal{A}_N  ( \psi_{2i+1-n,y} \mathcal{J}_{n,y} \partial_x^2  u ) - \psi_{2i+1-n,y} \mathcal{J}_{n,y} \partial_x^2  u \|_{H^2} \\
&\lesssim  N^{-2k+2}  \|  \psi_{2i+1-n,y} \mathcal{J}_{n,y} \partial_x^2  u \|_{H^{2k}} \\
&\lesssim  N^{-2k+2}   \|  \psi_{2i+1-n,y} \mathcal{J}_{n-2,y}^{\mathrm{main}}  u \|_{H^{2k}}  +  N^{-2k+2}   \|  \psi_{2i+1-n,y} \mathcal{J}_{n,y}^{\mathrm{bound}} \partial_x^2  u \|_{H^{2k}}  \\
&\lesssim  N^{-2k+2} \|  u \|_{H^{2k}}  \lesssim N^{-2k+2} \|  v \|_{H^{2k}} .
\end{split}
\end{equation*}
On the other hand, we set $f = \psi_{2i+1-n,y} \mathcal{J}_{n,y} u$, and we note that $\| f\|_{H^{2k+2}} \lesssim \| u \|_{H^{2k}}$.
Moreover, since $\Pi_{\mathcal{P}_N} = \mathcal{A}_N \Pi_{\mathcal{P}_N} $, for all $s\in \{0,2\}$, we have by Lemma \ref{lem:bound_A}
\begin{equation*}
\begin{split}
\| [\partial_x^2, \mathcal{A}_N] f \|_{H^s}  &= \| \partial_x^2 \mathcal{A}_N (\mathrm{Id} - \Pi_{\mathcal{P}_N}) f + \mathcal{A}_N (\Pi_{\mathcal{P}_N} - \mathrm{Id})\partial_x^2 f\|_{H^s} \\
&\lesssim \| \mathcal{A}_N  (\mathrm{Id} - \Pi_{\mathcal{P}_N}) f \|_{H^{s+2}} + \| (\Pi_{\mathcal{P}_N} - \mathrm{Id})\partial_x^2 f\|_{H^s}\\
& \lesssim N^{-(4-s-2)} \| f\|_{H^4}  \lesssim N^{-(4-s-2)}\|u\|_{H^2},
\end{split}
\end{equation*}
and thus, by Plancherel,
$$
\| [\partial_x^2, \mathcal{A}_N] f \|_{H^2_N} \lesssim  \|u\|_{H^2}.
$$
Finally, observing that $\partial_x^2 f$ vanishes at the boundary, we obtain
\begin{equation*}
\begin{split}
\| \partial_x^2 \mathcal{A}_N f -\partial_x^2 f \|_{H^2_N} &\leq \| [\partial_x^2, \mathcal{A}_N] f \|_{H^2_N} + \| (\mathcal{A}_N - \mathrm{Id})\partial_x^2 f \|_{H^2_N} \\
&= \| [\partial_x^2, \mathcal{A}_N] f \|_{H^2_N} + \| (\mathcal{A}_N - \mathrm{Id}) \partial_x^2 f \|_{H^2} \lesssim N^{-2k+2}\| u\|_{H^2},
\end{split}
\end{equation*}
which concludes the proof.
\end{proof}
As a corollary, we deduce the following useful result.
\begin{corollary} \label{cor:consis_crochet}
For all $N\geq 1$ and all $v\in H^{2k}(\mathbb{T})$ and all $v_N \in \mathcal{P}_N$, we have
$$
\| [\partial_x^2,\mathfrak{E}_{k,N}]v_N - [\partial_x^2,\mathfrak{E}_{k}]v  \|_{H^2_N} \lesssim \| v-v_N\|_{H^2_N} + N^{-2k+2}\| v \|_{H^{2k}}.
$$
\end{corollary}

\begin{proof}
We set $w_N = \mathcal{A}_N v$ and we decompose the bracket as
\begin{align*}
&\| [\partial_x^2,\mathfrak{E}_{k,N}]v_N - [\partial_x^2,\mathfrak{E}_{k}]v  \|_{H^2_N} \\
& \leq \| [\partial_x^2,\mathfrak{E}_{k,N}](v_N - w_N)  \|_{H^2_N}  + \| [\partial_x^2,\mathfrak{E}_{k,N}- \mathfrak{E}_{k}] w_N    \|_{H^2_N} + \|  [\partial_x^2,\mathfrak{E}_{k}](w_N - v)   \|_{H^2_N} .
\end{align*}
First, we control the last term. We note that $[\partial_x^2,\mathfrak{E}_{k}](w_N - v)  $ vanishes at the boundary. Thus, its $H^2_N$-norm coincide with its $H^2$-norm. Then we use that $[\partial_x^2,\mathfrak{E}_{k}]$ is bounded on $H^2$, see Lemma \ref{lem:bound_bracket_per}. Furthermore, we apply Lemma \ref{lem:consis_AN} to get
$$
\|  [\partial_x^2,\mathfrak{E}_{k}](w_N - v)   \|_{H^2_N} = \|  [\partial_x^2,\mathfrak{E}_{k}](w_N - v)   \|_{H^2} \lesssim  \|  w_N- v   \|_{H^2} \lesssim N^{-2k+2}\| v\|_{H^{2k}}.
$$
Then we focus on the first term. We recall that by Lemma \ref{lem:consis_crochet} we have the boundedness of $[\partial_x^2,\mathfrak{E}_{k,N}]$ on $H^2_N$. Thus by Lemma \ref{lem:consis_AN}, we have that
$$
\| [\partial_x^2,\mathfrak{E}_{k,N}](v_N - w_N)  \|_{H^2_N} \lesssim \| v_N - w_N\|_{H^2_N}  \lesssim  \| v_N - v\|_{H^2_N} +  N^{-2k+2}\| v\|_{H^{2k}}.
$$
Finally, the bound on the second term is given by Lemma \ref{lem:consis_crochet}, since there holds $\| w_N \|_{H^{2k}} \lesssim \|v\|_{H^{2k}}$ by Lemma \ref{lem:bound_A}.
\end{proof}
Finally, we deduce in the following result and its corollary the consistency and stability of $W_{k,N}^{\mathrm{cor}}$.
\begin{proposition} \label{prop:consis_W}
For all $N\geq 1$, all $v\in H^{2k}_{odd}(\mathbb{T})$ and all $v_N \in \mathcal{P}_N$, we have
$$
\| W_{k,N}^{\mathrm{cor}}v_N - W_{k}^{\mathrm{cor}} v  \|_{H^2} \lesssim \| v-v_N\|_{H^2_N} + N^{-2k+2}\| v \|_{H^{2k}}.
$$
\end{proposition} 
\begin{proof}
Since
$$
\partial_t \left( e^{- t\mathfrak{E}_{k,N}  } \partial_x^2 e^{t \mathfrak{E}_{k,N}  }  \right) = e^{- t\mathfrak{E}_{k,N}  } [\partial_x^2,\mathfrak{E}_{k,N}] e^{t \mathfrak{E}_{k,N}  },
$$
 we can write
$$
W_{k,N}^{\mathrm{cor}} = \Lambda_N  \Gamma_{1,N}+  \Lambda_N \Gamma_{2,N} \quad \text{and} \quad 
W_{k}^{\mathrm{cor}} =\Lambda \Gamma_1 +  \Lambda\Gamma_2,
$$
where we use the notation
$$
\Gamma_{1,N} = \int_{0}^1  e^{- t\mathfrak{E}_{k,N}  } [\partial_x^2,\mathfrak{E}_{k,N}] e^{t \mathfrak{E}_{k,N}  } \mathrm{d}t, \quad \Gamma_{2,N}= e^{- \mathfrak{E}_{k,N}  }  \mathcal{A}_N (W\cdot)  e^{ \mathfrak{E}_{k,N}  }, 
$$
and similarly,
$$
\Gamma_1= \int_{0}^1  e^{- t\mathfrak{E}_{k}  } [\partial_x^2,\mathfrak{E}_{k}] e^{t \mathfrak{E}_{k}  } \mathrm{d}t, \quad \Gamma_2=e^{- \mathfrak{E}_{k}  } W  e^{ \mathfrak{E}_{k}  }. 
$$
Thus by Lemma \ref{lem:consis_lambda}, we have
$$
\| W_{k,N}^{\mathrm{cor}}v_N - W_{k}^{\mathrm{cor}} v  \|_{H^2} \lesssim \| \Gamma_{1,N} v_N - \Gamma_{1} v \|_{H^2_N} +  \| \Gamma_{2,N} v_N - \Gamma_{2} v \|_{H^2_N} + N^{-2k+2}\| (\Gamma_{1} + \Gamma_2) v \|_{H^{2k}}.
$$
By Lemmata \ref{lem:bound_E_per} and \ref{lem:bound_bracket_per}, we have
$
\| (\Gamma_{1} + \Gamma_2) v \|_{H^{2k}} \lesssim \| v\|_{H^{2k}}.
$
Then, we focus on estimating $\| \Gamma_{1,N} v_N - \Gamma_{1} v \|_{H^2_N}$ (the bound on $\| \Gamma_{2,N} v_N - \Gamma_{2} v \|_{H^2_N}$ is similar and does not require extra arguments, so we skip it).

We fix $t\in[0,1]$ and aim to prove uniform bounds in $t$.  Applying Lemma \ref{lem:consis_exp_E} and \ref{cor:consis_crochet}, we have
\begin{equation*}
\begin{split}
 & \| e^{- t\mathfrak{E}_{k}  } [\partial_x^2,\mathfrak{E}_{k}] e^{t \mathfrak{E}_{k}  } v-  e^{- t\mathfrak{E}_{k,N}  } [\partial_x^2,\mathfrak{E}_{k,N}] e^{t \mathfrak{E}_{k,N}  }v_N \|_{H^2_N}\\
 & \lesssim \| v-v_N \|_{H^2_N} + N^{-2k +2} ( \| [\partial_x^2,\mathfrak{E}_{k}] e^{t \mathfrak{E}_{k}  } v \|_{H^{2k}} + \|  e^{t \mathfrak{E}_{k}  } v \|_{H^{2k}} + \|  v \|_{H^{2k}}). \\
\end{split}
\end{equation*}
Thus, applying Lemmata \ref{lem:bound_E_per} and \ref{lem:bound_bracket_per}, we get
$$
\| e^{- t\mathfrak{E}_{k}  } [\partial_x^2,\mathfrak{E}_{k}] e^{t \mathfrak{E}_{k}  } v-  e^{- t\mathfrak{E}_{k,N}  } [\partial_x^2,\mathfrak{E}_{k,N}] e^{t \mathfrak{E}_{k,N}  }v_N \|_{H^2_N} \lesssim \| v-v_N \|_{H^2_N} + N^{-2k +2}  \|  v \|_{H^{2k}},
$$
from what we deduce the desired estimate.
\end{proof}

\begin{corollary} \label{cor:bound_W_N} For all $N\geq 1$, all $1\leq a\leq k$, all $u \in \Lambda_N \mathcal{P}_N= \mathcal{P}_N \cap H^\infty_{odd}$, we have
$$
\| W_{k,N}^{\mathrm{cor}} u  \|_{H^{2a}} \lesssim  \| u \|_{H^{2a}}.
$$
\end{corollary}
\begin{proof} First, we recall that by Lemma \ref{lem:alg_Lambda}, $\Lambda_N \mathcal{P}_N = \mathcal{P}_N \cap H^\infty_{odd}$. Then we apply Proposition \ref{prop:consis_W} with $v=v_N =u$ and, using Corollary \ref{cor:bound_W} , we get that
\begin{equation*}
\begin{split}
\| W_{k,N}^{\mathrm{cor}} u  \|_{H^{2a}}  &\leq \|W_{k,N}^{\mathrm{cor}} u  -  \Pi_{\mathcal{P}_N} W_{k}^{\mathrm{cor}} u  \|_{H^{2a}} + \| \Pi_{\mathcal{P}_N} W_{k}^{\mathrm{cor}} u  \|_{H^{2a}}  \\
&\leq N^{2(a-1)} \|W_{k,N}^{\mathrm{cor}} u  -  \Pi_{\mathcal{P}_N} W_{k}^{\mathrm{cor}} u  \|_{H^{2}} +  \| W_{k}^{\mathrm{cor}} u  \|_{H^{2a}} \\
&\leq N^{2(a-1)} \|W_{k,N}^{\mathrm{cor}} u  -   W_{k}^{\mathrm{cor}} u  \|_{H^{2}} +  \| W_{k}^{\mathrm{cor}} u  \|_{H^{2a}} \\
&\lesssim N^{-2(k-a)} \|u \|_{H^{2k}} + \| u\|_{H^{2a}} \lesssim \| u\|_{H^{2a}},
\end{split}
\end{equation*}
which is the desired result.
\end{proof}

\subsection{Error estimates of arbitrarily high order} \label{sec:errorest}
We are now in the position to state the main result of this paper. Somehow, we want to merge Proposition \ref{prop:cvg_space} about the space discretization and Proposition \ref{prop:disc_time} about the time discretization. Thus,
naturally, we define the full discretization by 
\begin{equation}
\label{eq:la_discretization}
u_{n,N} =    e^{\mathfrak{E}_{k,N}} (\Phi^{\mathrm{cor}}_{\tau, N})^n \Lambda_N  e^{-\mathfrak{E}_{k,N}}  w^{(0)}_N
\end{equation}
where
$$
\Phi^{\mathrm{cor}}_{\tau,N}=e^{-\ic b_s \tau \partial_x^2}  r(-\ic a_s \tau W^{\mathrm{cor}}_{k,N}) \cdots  e^{-\ic b_1 \tau \partial_x^2}   r(-\ic a_1 \tau W^{\mathrm{cor}}_{k,N})
$$
and, as in Subsection \ref{sub:time_disc}, $a_1, b_1, \cdots,  a_s, b_s\in \mathbb{R}$ are the coefficients of a splitting method of order $p\geq 1$, $s \geq 1$ is the number of stages  and $r$ is a real analytic function defined on a neighbourhood of the origin and such that
$$
r(z) = e^z + \mathcal{O}(z^{p+1}) \quad \text{as} \; z \to 0.
$$
Then, we prove the convergence of this full discretization.
\begin{theorem} \label{thm:conv}
If $1\leq a \leq k$ and $u^{(0)} \in H^{2k}_V$ then for all $T>0$ and $N\geq 1$, we have 
$$
\| u(t_n) - \mathcal{R} u_{n,N}\|_{H^{2a}} \lesssim_{T,k} (\tau^{\min(k-a,p)} + N^{-2(k-a)}) \| w^{(0)} \|_{H^{2k}}\quad \text{for}\; 0 \leq t_n \leq T,
$$
where $u\in C^0([0,T];H^{2k}_V)$ is the solution of \eqref{eq:ev_Dir} and $(u_{n,N})_{n\geq 0} \in \mathcal{P}_N^{\mathbb{N}}$ is the full discretization given by \eqref{eq:la_discretization}.
\end{theorem}

 In order to prove Theorem~\ref{thm:conv}, we show first convergence of the space discretization in the new variables, and afterwards Proposition~\ref{prop:cvg_space}.

\begin{lemma} \label{lem:consis_flow}For all $N\geq 1$, all $v\in H^{2k}_{odd}(\mathbb{T})$, all $v_N \in  \mathcal{P}_N \cap H^\infty_{odd}$, all $T>0$ and all $t\in [-T,T]$, we have
$$
\|w_N(t) -w(t)  \|_{H^2} \lesssim_T \| v-v_N \|_{H^2}  + N^{-2k+2} \| v \|_{H^{2k}},
$$
where we set $w_N(t) = e^{-\ic t (\partial_x^2+W_{k,N}^{\mathrm{cor}})} v_N$ and $w(t) = e^{-\ic t (\partial_x^2+W_{k}^{\mathrm{cor}})} v$.
\end{lemma}
\begin{proof} Without loss of generality, we assume $t\geq 0$.
By Duhamel, we have
$$
w(t) - w_N(t) = e^{-\ic t \partial_x^2} (v-v_N) + \int_{0}^t e^{-\ic (t-\tau) \partial_x^2} (W_{k}^{\mathrm{cor}} w(\tau) - W_{k,N}^{\mathrm{cor}} w_N(\tau)) \, \mathrm{d}\tau.
$$
Using that $e^{-\ic t \partial_x^2}$ is an isometry on $H^2$, we get by Proposition \ref{prop:consis_W}
$$
\| w(t) - w_N(t) \|_{H^2} \lesssim  \| v-v_N \|_{H^2}   + \int_{0}^t \| w(\tau) - w_N(\tau) \|_{H^2} \, \mathrm{d}\tau + N^{-2k+2} T \sup_{0\leq \tau \leq T} \| w(\tau) \|_{H^{2k}}.
$$
Then, we use that $W_{k}^{\mathrm{cor}}$ is bounded on $H^{2k}_{odd}$ (Corollary \ref{cor:bound_W}) to get that $\| w(\tau) \|_{H^{2k}} \lesssim_{T}  \| v \|_{H^{2k}}$.
Finally, we conclude by applying the Gr\"onwall inequality.
\end{proof}

\begin{ProofOf}{Proposition~\ref{prop:cvg_space}}
First let us prove that it suffices to deal with the case $a=1$: we assume that \eqref{eq:lapin} holds for $a=1$ and we prove it for $a>1$. We use the notation given in~\eqref{def:v_vN} for $v(t)$ and $v_N(t)$.
We have
\begin{equation*}
\begin{split}
\| v(t) - v_N(t) \|_{H^{2a}} &\leq \| (\mathrm{Id} - \Pi_{\mathcal{P}_N})v(t) \|_{H^{2a}} + \|  \Pi_{\mathcal{P}_N} v(t) -v_N(t) \|_{H^{2a}} \\
&\leq  N^{-2(k-a)}\| v(t) \|_{H^{2k}} + N^{2(a-1)} \| \Pi_{\mathcal{P}_N} v(t) -v_N(t) \|_{H^{2}} \\
&\leq N^{-2(k-a)}\| v(t) \|_{H^{2k}} + N^{2(a-1)} \|  v(t) -v_N(t) \|_{H^{2}} \\
&\mathop{\lesssim}^{\eqref{eq:lapin}}  N^{-2(k-a)}\| v(t) \|_{H^{2k}} + N^{2(a-1)} N^{-2k+2} \| w^{(0)}\|_{H^{2k}}.
\end{split}
\end{equation*}

Finally, it suffices to prove that $\| v(t) \|_{H^{2k}} \lesssim_T \| w^{(0)} \|_{H^{2k}}$. Since $\mathfrak{E}_k $ is bounded on $H^{2k}$ (Lemma~\ref{lem:bound_E_per}), and $W_k^{\mathrm{cor}}$ is bounded on $H^{2k}_{odd}$ (Corollary \ref{cor:bound_W}), we have 
\begin{equation} \label{eq:pouette1}
\| v(t) \|_{H^{2k}} \lesssim \|  e^{-\ic t (\partial_x^2+W_k^{\mathrm{cor}})} \Lambda e^{ -\mathfrak{E}_k  } w^{(0)} \|_{H^{2k}} \lesssim_T \|  \Lambda e^{ -\mathfrak{E}_k  } w^{(0)} \|_{H^{2k}}.
\end{equation}
 Furthermore, note that $\mathcal{R} w^{(0)} \in H^{2k}_V$, and therefore 
$$ 
\mathcal{R} e^{ -\mathfrak{E}_k  } w^{(0)} = e^{ -\mathfrak{C}_k  }  \mathcal{R} w^{(0)}  \in H^{2k}_{\mathrm{Dir}}
$$ 
due to Lemma~\ref{lem:pas_si_mal}. Thus, by Lemma \ref{lem:bound_lambda_per},
\begin{equation} \label{eq:pouette2}
\| \Lambda e^{ -\mathfrak{E}_k  } w^{(0)} \|_{H^{2k}} \lesssim \|  e^{ -\mathfrak{E}_k  } w^{(0)} \|_{H^{2k}} \lesssim \| w^{(0)} \|_{H^{2k}}.
\end{equation} 

Finally, to get \eqref{eq:lapin} in the case $a=1$, it suffices to apply successively Lemmata \ref{lem:consis_lambda}, \ref{lem:consis_exp_E}, \ref{lem:consis_flow} in order to control the $H^{2k}$-norms arising in \eqref{eq:pouette1}-\eqref{eq:pouette2}.
\end{ProofOf}

\begin{ProofOf}{Theorem~\ref{thm:conv}} First, by Proposition \ref{prop:period_ext}, it suffices to prove that
$$
 \| v(t_n) - u_{n,N}\|_{H^{2a}} \lesssim_{T,k} (\tau^{\min(k-a,p)} + N^{-2(k-a)}) \| w^{(0)} \|_{H^{2k}}\quad \text{for}\; 0 \leq t_n \leq T,
$$
where
$$
v(t) = e^{ \mathfrak{E}_k  } e^{-\ic t (\partial_x^2+W_k^{\mathrm{cor}})} \Lambda e^{ -\mathfrak{E}_k  } w^{(0)}.
$$
Applying Proposition \ref{prop:cvg_space} about the semi-discretization in space, we have
\begin{equation}
\label{eq:coquin}
 \| v(t) -  e^{ \mathfrak{E}_{k,N}  } e^{-\ic t (\partial_x^2+W_{k,N}^{\mathrm{cor}})} z_N^{(0)}  \|_{H^{2a}}  \lesssim_{T}  N^{-2(k-a)} \|w^{(0)} \|_{H^{2k}},
\end{equation}
where we denote
$$
z_N^{(0)} = \Lambda_N e^{ -\mathfrak{E}_{k,N}  } w_N^{(0)}.
$$
Then, using that $\mathfrak{E}_{k,N}$ is bounded (uniformly in $N$) in $H^{2a}$ (it is a direct corollary of Lemmata \ref{lem:bound_A} and \ref{lem:diam}), it suffices to prove the convergence of the time discretization of $\exp(-\ic t (\partial_x^2+W_{k,N}^{\mathrm{cor}}))$, i.e. we have
\begin{equation*}
\|e^{-\ic t_n (\partial_x^2+W_{k,N}^{\mathrm{cor}})} z_N^{(0)} - (\Phi^{\mathrm{cor}}_{\tau, N})^n z_N^{(0)}  \|_{H^{2a}} \lesssim_{T} \tau^{\min(k-a,r)}  \| w^{(0)} \|_{H^{2k}} \quad \text{for}\; 0 \leq t_n \leq T.
\end{equation*}
Since $W_{k,N}^{\mathrm{cor}}$ is bounded from $H^{2j}_{odd} \cap \mathcal{P}_N$ to $H^{2j}_{odd} \cap \mathcal{P}_N$ (uniformly in $N$) for all $j\in \{1,\ldots,N\}$ (Corollary \ref{cor:bound_W_N}), applying the classical convergence result for splitting methods \cite{Jahnke:2000:EBF}, we get that
\begin{equation*}
\|e^{-\ic t_n (\partial_x^2+W_{k,N}^{\mathrm{cor}})} z_N^{(0)} - (\Phi^{\mathrm{cor}}_{\tau, N})^n z_N^{(0)}  \|_{H^{2a}} \lesssim_{T} \tau^{\min(k-a,r)}  \| z_N^{(0)} \|_{H^{2k}} \quad \text{for}\; 0 \leq t_n \leq T.
\end{equation*}
As a consequence, it suffices to prove that $\| z_N^{(0)} \|_{H^{2k}} \lesssim  \| w^{(0)} \|_{H^{2k}}$. Indeed, since $\mathfrak{E}_{k,N}$ is bounded (uniformly in $N$) in $H^{2k}$, we have by \eqref{eq:coquin} (applied in $t=0$ with $a=k$),
$$
 \| z_N^{(0)}  \|_{H^{2k}} \lesssim \| e^{ -\mathfrak{E}_{k,N}  } v(0) \|_{H^{2k}} +  \|w^{(0)} \|_{H^{2k}} \lesssim \| e^{ \mathfrak{E}_{k}  } \Lambda e^{ -\mathfrak{E}_{k}  }  w^{(0)} \|_{H^{2k}} +  \|w^{(0)} \|_{H^{2k}}.
$$
Finally, since 
$$
\mathcal{R}e^{ -\mathfrak{E}_{k}  }  w^{(0)}= e^{ -\mathfrak{C}_{k}  }  u^{(0)} \in H^{2k}_{\mathrm{Dir}}
$$ 
due to Lemma \ref{lem:pas_si_mal}, we have by Lemma \ref{lem:bound_lambda_per} (and using that $\mathfrak{E}_{k}$ is bounded on $H^{2k}$) that
$$
\| e^{ \mathfrak{E}_{k}  } \Lambda e^{ -\mathfrak{E}_{k}  }  w^{(0)} \|_{H^{2k}} \lesssim   \|w^{(0)} \|_{H^{2k}},
$$
and thus, the desired error estimate holds true.
\end{ProofOf}

\section{Numerical results} \label{sec:numerics}
\sectionmark{\footnotesize  Numerical results}

In this section, we illustrate the numerical time integration of the linear Schrödinger equation
\begin{equation} \label{eq:ev_num}
\ic \partial_t u(t,x) = (\partial_x^2 +V(x))u(t,x) \quad \text{in}\; [0,T]\times (0,L), \qquad u(t,0) = u(t,L)=0 \quad \text{in}\; [0,T], 
\end{equation}
with initial condition $u^{(0)}(x)=u(0,x)$ by means of splitting methods. Thereby, we consider the subproblems~\eqref{prob:lapl} and~\eqref{prob:pot}, with exact flows $e^{-\ic t\partial^2_x}u^{(0)}$ and $e^{-\ic tV}u^{(0)}$ for $t \in [0, T]$.

\begin{remark} \label{rem:L}
 While we take $x \in (0,1)$ for the convergence analysis, we numerically integrate problem~\eqref{eq:ev_num} over a larger interval $(0, L)$, with $L>1$. As explained in the previous section, in order to define the corrected potential $W^{\mathrm{cor}}_k$ on the torus $ \mathbb{T} = \mathbb{R}/2\mathbb{Z}$, we require periodic extensions of the initial condition $u^{(0)} \mapsto w^{(0)}$, the potential $V \mapsto W$, as well as the functions $\varphi_{n,y} \mapsto \psi_{n,y}$, $ y \in \{0,1\}$. Classically, this is achieved by means of so-called bump functions, which ensure a smooth, periodic extension. However, derivatives of these extended functions may become very large, leading to a significant global error between the exact solution and the modified schemes. Thus, even though high order convergence still persists in this case, we may observe unfavorable error constants. To avoid such large constants, we choose $L>1$ and consider smooth functions $w^{(0)}, W, \psi_{n,y}$, $ y \in \{0,L\}$, on the larger torus $\mathbb{T}_L = \mathbb{R}/2L\mathbb{Z}$, and work with their restrictions to the interval $[0, L]$. Analogously, we consider the function spaces $H^{2k}$, $H^{2k}_{\mathrm{Dir}}$ and $H^{2k}_V$ over $[0,L]$.
\end{remark}

In this section, we numerically integrate in time the boundary value problem~\eqref{eq:ev_num} for various potentials $V$ and initial conditions $u^{(0)}$. Table~\ref{table:reg} summarizes all data considered, together with the compatibility conditions satisfied by the initial condition for each potential. In addition, we indicate in which figure the corresponding functions are used in the numerical experiments.

\begin{table}[H] \label{table:reg}
\captionsetup{labelformat=simple,name=Table,labelfont=bf,textfont=it}
\caption{Regularity of the initial condition $u^{(0)}$ \label{table:reg}}
\centering

\begin{tabular}{p{0.25\textwidth}|p{0.2\textwidth}|p{0.1\textwidth}|p{0.29\textwidth}}
\toprule
 Initial condition & Potential & Regularity & Figure \\
 \midrule
 $u_1^{(0)}(x)=\sin(\frac{2\pi x}{L})$ & $V_1(x)=\cos(\frac{2\pi x}{L})$ & $H^{\infty}_{V_1}$ & Fig.~\ref{fig:naiv} (left) \\

&&\\[-2.5ex]
\hline
&&\\[-2.5ex]

 \multirow{3}{*}{$u_2^{(0)}(x)=x(L-x)e^{\frac{x}{L}-\frac{x^2}{L^2}} $}
& $V_2(x)=1+\frac{4x}{L^3}-\frac{4x^2}{L^4}$ & $H^{\infty}_{V_2}$ & Fig.~\ref{fig:naiv} (right), Fig.~\ref{fig:corr1} (left) \\

&&\\[-2.5ex]
\cline{2-4}
&&\\[-2.5ex]

& $V_3(x)=\sin(\frac{2\pi x}{L})$ & $H^4_{V_3}$ &  Fig.~\ref{fig:corr1} (right) \\

&&\\[-2.5ex]
\cline{2-4}
&&\\[-2.5ex]

  & $V_4(x)=e^{\frac{x}{L^2}}$  & $H^4_{V_4}$ & Fig.~\ref{fig:corr2} \\
\bottomrule
\end{tabular}
\end{table}

For the illustration of the results in the previous sections, we focus on the symetric splitting scheme~\eqref{y4} with real coefficients. Recall that the method is of formal order four, but converges in general with reduced order when applied to problem~\eqref{eq:ev_num}, in contrast to the second order Strang splitting method~\eqref{strang},
which does not suffer from an order reduction, see Figure~\ref{fig:naiv}.

\begin{remark} \label{rem:oddeven}
If an odd initial condition $w^{(0)} \in H^{2k}$ and an even potential $W$ are considered on the torus $ \mathbb{T}_L$, then both, the Laplacian $u \mapsto \partial^2_xu$ and the potential $u \mapsto Wu$, commute with the projector $\Lambda$ onto the subspace of odd periodic functions $H^{2k}_{\text{odd}}(\mathbb{T}_L)$. Hence, the exact flows $e^{-\ic t\partial^2_x}$ and $e^{-\ic tV}$ preserve $H^{2k}_{\text{odd}}(\mathbb{T}_L)$, and the numerical solutions obtained by the splitting methods~\eqref{strang} and~\eqref{y4} on the torus coincide (after restriction to the interval $(0, L)$) with the splitting solutions of problem~\eqref{eq:ev_num}~\cite{Bona:2018:NBV}. With $u^{(0)}_1$ and $V_1$ we are in this setting, and therefore the naive schemes converge with full order, whereas the mentioned regularity is not satisfied for $u^{(0)}_2$ and $V_2$.
\end{remark}

In the following, we denote the naive splitting schemes without any corrections by \emph{Strang} for~\eqref{strang} and \ynull for method~\eqref{y4}. Moreover, \yk refers to the family of corrected methods~\eqref{eq:la_discretization}, depending on the parameter $k=2, 3, 4$, where we use the identities in Example~\ref{ex:c2c3c4} to compute the correctors $\mathfrak{E}_{2,N}, \mathfrak{E}_{3,N}$ and $\mathfrak{E}_{4,N}$. The flows $e^{\pm \mathfrak{E}_{k,N}}$ are computed by one step of an explicit fourth order Runge-Kutta method at time $t=1$. The method was chosen out of simplicity, there are other possibilities to approximate the exact flow of the corrector functions. In general, the computational costs of this flows can be very big for an approximation with high precision. However, for the considered examples, we do not observe a significant improvement in terms of the error constant if we choose more steps in order to compute $e^{\pm \mathfrak{E}_{k,N}}$. As reference solution at final time $T$, we use the same Runge-Kutta method, here with a very small time step $\tau_{ref} = 10^{-6}$. We choose $N=512$ points to discretize the interval $[0,L]$ with $L=2\pi$, and we proceed the simulation for time steps $\tau = 0.02\cdot 2^{-n}, n= 0, \ldots, 5$, and a final time $T=0.1$. The \textsc{Matlab} codes which compute these modified splitting schemes and illustrates their convergence are available at \cite{codesMatlab26}.

\begin{figure}[!tbp]
\begin{minipage}[c]{0.48\textwidth}
\includegraphics[width=\textwidth]{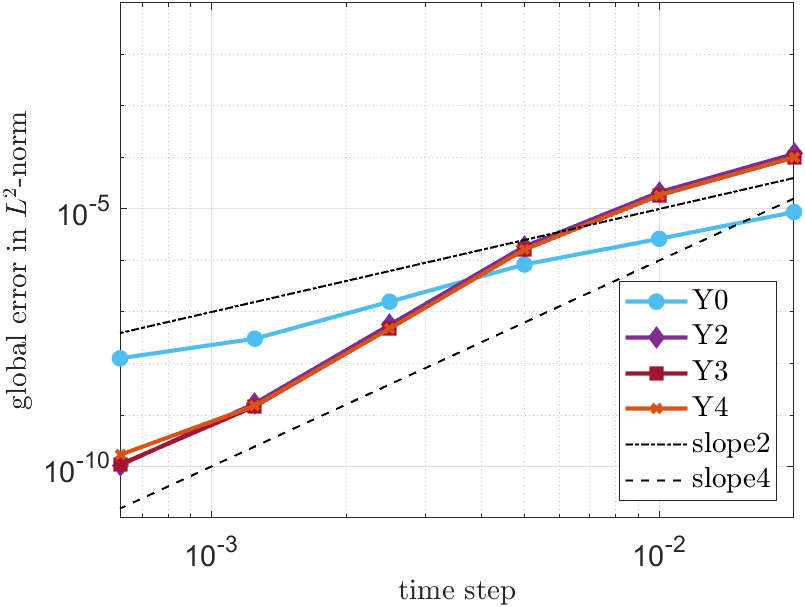}
\end{minipage}
\quad
\begin{minipage}[c]{0.48\textwidth}
\includegraphics[width=\textwidth]{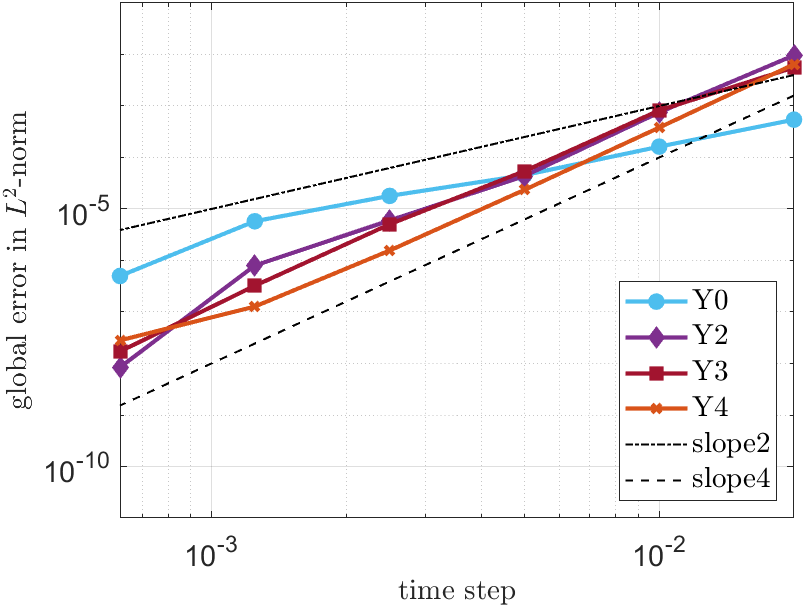}
\end{minipage}
\vspace*{1mm}
\caption{\small Convergence error of the corrected splitting schemes \yk, $k=2, 3, 4$ applied to problem~\eqref{eq:ev_num} for two different potentials $V_2$ (left) and $V_3$ (right), and initial condition $u^{(0)}_2$. Comparison to the naive method \ynull. Reference slopes are given in dashed lines. }
\label{fig:corr1}
\end{figure}

\begin{figure}[!tbp]
\begin{minipage}[c]{0.48\textwidth}
\includegraphics[width=\textwidth]{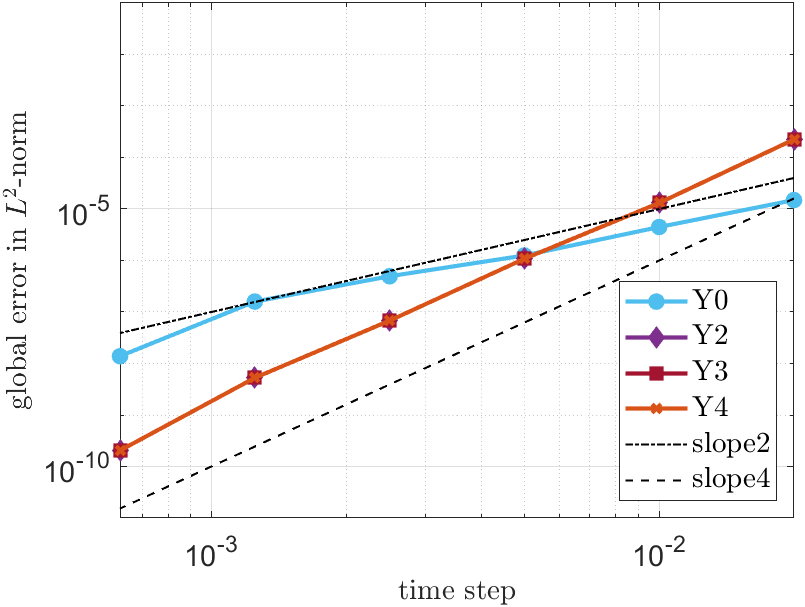}
\end{minipage}
\quad
\begin{minipage}[c]{0.48\textwidth}
\includegraphics[width=\textwidth]{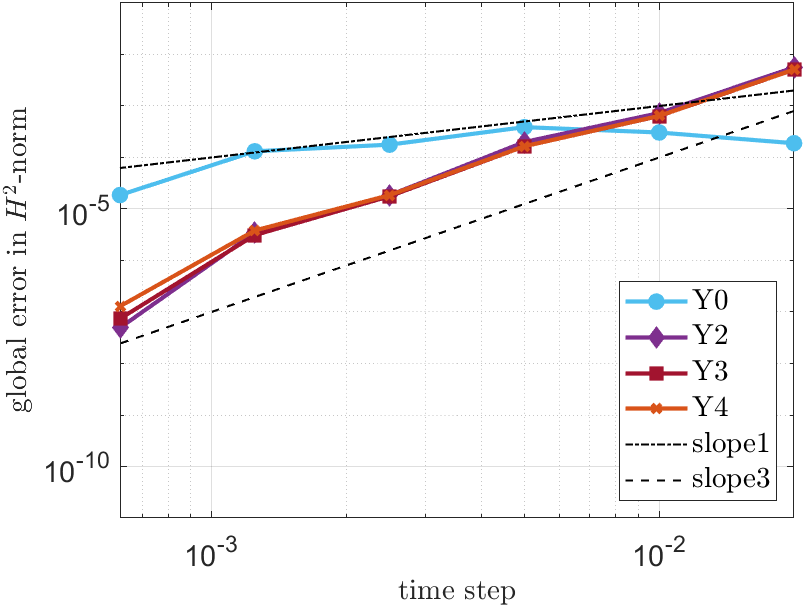}
\end{minipage}
\vspace*{1mm}
\caption{\small Convergence error of the corrected splitting schemes \yk, $k=2, 3, 4$ applied to problem~\eqref{eq:ev_num} for the potential $V_4$ and initial condition $u^{(0)}_2$. Comparison to the naive method \ynull. The global error is plotted in the $L^2$-norm (left) and the $H^2$-norm (right). Reference slopes are given in dashed lines.}
\label{fig:corr2}
\end{figure}

In the following two Figures, we consider the initial condition $u_2^{(0)}$, whose regularity depends on the potential, see Table~\ref{table:reg}. In view of Theorem~\ref{thm:conv}, we expect the global error of \yk, $k=2,3,4$, to be of order $\mathcal{O}(\tau^{k-1})$ in the $L^2$-norm, whenever $u_2^{(0)} \in H^{2k}_V$. 

In Figure~\ref{fig:corr1}, we compare the convergence of the naive scheme \ynull with the one of its modifications \yzwei, \ydrei and \yvier for two potentials $V_2$ and $V_3$. We have $u_2^{(0)} \in H^{\infty}_{V_2} \subset H^{8}_{V_2}$, so, by the convergence analysis, we expect a third order convergence for the scheme \yvier. However, numerically we observe that the corrected method converges with order four for almost all considered time steps $\tau$. It seems that this behavior persists for the schemes \yzwei and \ydrei. Furthermore, note that for $V_3$, the initial conditions has lower regularity, namely $u_2^{(0)} \in H^{4}_{V_3}$, than what is theoretically required for a full order convergence. As before, we observe order reduction for \ynull, whereas all corrected methods achieve full fourth order convergence. Moreover, for both considered potentials, the modified methods outperform the naive scheme also in terms of the error constant for small time steps $\tau$. 

In Figure~\ref{fig:corr2}, we consider the potential $V_4$, which is, in contrast to $V_2$ and $V_3$, non-symmetric with respect to the middle point $L/2$ of the interval $[0,L]$, and again we have only $u_2^{(0)} \in H^{4}_{V_4}$. Nevertheless, for all values $k= 2, 3, 4$, the corrected methods still exhibit full order convergence, while the naive scheme \ynull suffers from an order reduction and converges with order 1.5. In the right panel, we additionally plot the error in the $H^2$-norm (which corresponds to $a=1$ in Theorem~\ref{thm:conv}), and observe third order convergence for all modified schemes, rather than only in the case $k=4$, as predicted by the theorem. Furthermore, for small time steps, the corrected methods again yield smaller error constants, regardless of which norm is used to measure the global error.

\begin{figure}[!tbp]
\begin{minipage}[c]{0.48\textwidth}
\includegraphics[width=\textwidth]{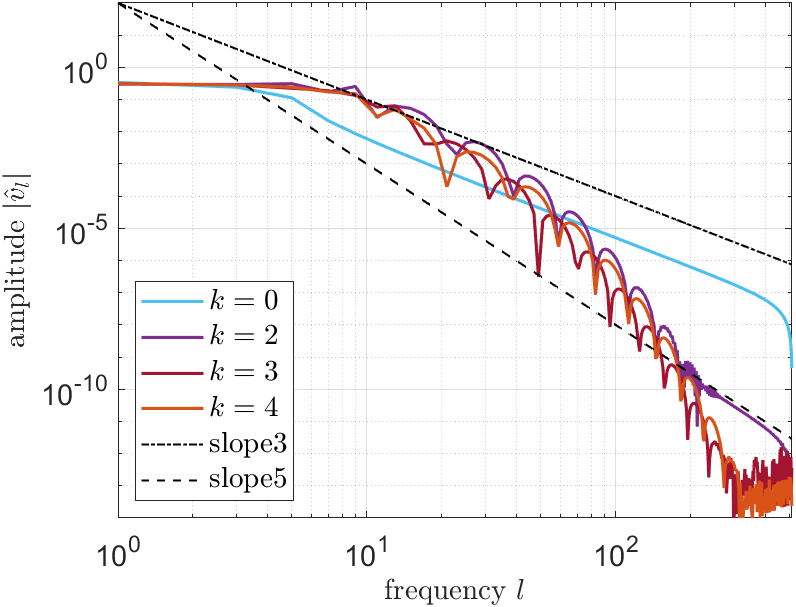}
\end{minipage}
\quad
\begin{minipage}[c]{0.48\textwidth}
\includegraphics[width=\textwidth]{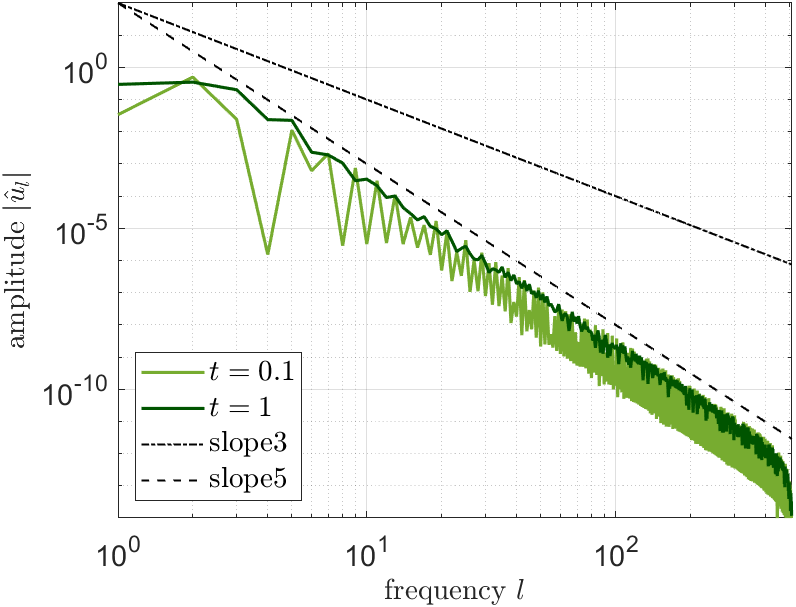}
\end{minipage}
\vspace*{1mm}
\caption{ \small Left: Regularity of the classical potential $V=V_3$ and its corrected version $V_k^{\text{cor}}$ applied to $u=u^{(0)}_1$. We plot the Fourier modes $|\hat{v}_{l}|$ of $\mathcal{A}_N(W_3w^{(0)}_1)$ ($k=0$) and $W_{k,N}^{\text{cor}}e^{-\mathfrak{E}_{k,N}}w^{(0)}_{1,N}$ for $k=2, 3, 4$. Right: Fourier modes $|\hat{u}_{l}|$ of the exact solution of the problem~\eqref{eq:ev_num} with $u^{(0)}_1, V_3$ at time $t=0.1$ and $t=1$. Reference slopes of order three and five are given in dashed lines.}
\label{fig:intro_schr_fm}
\end{figure}

\begin{remark} \label{rem:regularityofschr}
The regularity of the corrected potential $W^{\mathrm{cor}}_{k,N}$ applied to some $v_N \in \mathcal{P}_N$ is usually observed only for high Fourier modes. However, in this regime round-off errors show up and thus, the difference in this regularity is often not visible for different values of $k$.
In order to illustrate this phenomenon, we plot the Fourier coefficients of $v_N=\mathcal{A}_N(W_3w^{(0)}_1)$ and the corresponding corrected potential $W_{k,N}^{\text{cor}}e^{-\mathfrak{E}_{k,N}}w^{(0)}_{1,N}$ for $k=2, 3, 4$ in Figure~\ref{fig:intro_schr_fm}. Thereby, we denote by $W_3$ and $w^{(0)}_1$ the extended potential $V_3$, initial condition $u^{(0)}_1$ respectively. Note that $v_N$ is even, and therefore all the $\sin$-terms vanish (up to machine precision), that is why we only show the corresponding $\cos$-terms. We observe that the modes of $v_N$ are of order three in terms of the frequency $l$, what confirms that the boundary conditions are not preserved by the classical splittings. In contrast, as proven in Theorem~~\ref{thm:main}, the corrected potential preserves the space $H^{2k}_{\mathrm{Dir}}$. Precisely, the expected regularity can be observed for $k=2$ for high modes. For $k=3,4$, round-off errors appear in this regime. In contrast, for lower modes we cannot really spot a difference between $W^{\mathrm{cor}}_ {2,N}, W^{\mathrm{cor}}_{3,N}$ and $W^{\mathrm{cor}}_{4,N}$. As illustrated in Figures~\ref{fig:corr1} and~\ref{fig:corr2}, the simplest correction, what corresponds to the case $k=2$, often avoids order reduction in the splitting scheme~\eqref{y4}. In addition, we observe that the Fourier coefficients of the exact solution are of order five in terms of $l$, which explains numerically that we gain two orders thanks to the dispersion $\partial^2_x$, see Figure~\ref{fig:intro_schr_fm} (right).
\end{remark}

To conclude, although the naive scheme \ynull has formal order four, it generally integrates the boundary value problem~\eqref{eq:ev_num} with a reduced convergence order between one and two. In contrast, the corrected scheme \yvier achieves at least third order convergence whenever $u^{(0)} \in H^8_V$, which confirms the statement of Theorem~\ref{thm:conv}. In practice, \yvier converges with full order four for many examples, which is finally better than what we expect from the convergence analysis. Numerically, this result persists for initial conditions which do not satisfy the compatibility conditions at sufficiently high order, and also for the methods \yzwei and \ydrei.

\section*{Conclusion}
\sectionmark{\footnotesize  References}

We introduced to the best of our knowledge the first corrector technique for splitting methods for the approximation of smooth solutions of the linear Schrödinger equation with non-trivial compatibility conditions, and obtained a family of corrected splitting methods achieving an arbitrary high order of accuracy. In contrast, note that the similar question of arbitrarily high order remains open in the context of splitting methods for parabolic problems with non-periodic boundary conditions.
Although the proposed implementation is mainly for validating the theory, we believe that this contribution paves the way for a more efficient implementation with possible future generalizations to multiple dimensions and beyond the linear case.

\section*{Acknowledgement}

 J.B. thanks Dario Bambusi and Dorian Le Peutrec for stimulating discussions about compatibility conditions. J.B. was partially supported by the ANR project KEN ANR-22-CE40-0016. R.H. and G.V. where partially supported by the Swiss National Science Foundation, projects No. 200020\_214819, No. 200020\_192129 and No. 10009199.

\bibliographystyle{plainurl}
\bibliography{ref_schroedinger}

\begin{thebibliography}{10}

\bibitem{Abou:2025:ACO}
C.~Abou~Khalil and J.~Bernier.
\newblock Almost conservation of the harmonic actions for fully discretized
  nonlinear {K}lein--{G}ordon equations at low regularity.
\newblock {\em IMA J. Num. Anal. draf098}, 2025.
\newblock \href {https://doi.org/10.1093/imanum/draf098}
  {\path{doi:10.1093/imanum/draf098}}.

\bibitem{Antoine:2008:ARO}
X.~Antoine, A.~Arnold, Ch. Besse, M.~Ehrhardt, and A.~Schädle.
\newblock A review of transparent and artificial boundary conditions techniques
  for linear and nonlinear {S}chrödinger equations.
\newblock {\em Commun. Comput. Phys.}, 4(4):729--796, 2008.
\newblock \href {https://doi.org/10.14279/depositonce-15662}
  {\path{doi:10.14279/depositonce-15662}}.

\bibitem{Antoine:2001:CSA}
X.~Antoine and Ch. Besse.
\newblock Construction, structure and asymptotic approximations of a
  microdifferential transparent boundary condition for the linear
  {S}chrödinger equation.
\newblock {\em J. Math. Pures Appl.}, 80(7):701--738, 2001.
\newblock \href {https://doi.org/10.1016/S0021-7824(01)01213-2}
  {\path{doi:10.1016/S0021-7824(01)01213-2}}.

\bibitem{Antoine:2012:ABC}
X.~Antoine, Ch. Besse, and P.~Klein.
\newblock Absorbing boundary conditions for the two-dimensional {S}chrödinger
  equation with an exterior potential. {P}art {I}: construction and \emph{a
  priori} estimates.
\newblock {\em Math. Models Methods Appl. Sci.}, 22(10), 2012.
\newblock \href {https://doi.org/10.1142/S0218202512500261}
  {\path{doi:10.1142/S0218202512500261}}.

\bibitem{Antoine:2013:ABC}
X.~Antoine, Ch. Besse, and P.~Klein.
\newblock Absorbing boundary conditions for the two-dimensional {S}chrödinger
  equation with an exterior potential. {P}art {II}: discretization and
  numerical results.
\newblock {\em Numer. Math.}, 125:191--223, 2013.
\newblock \href {https://doi.org/10.1007/s00211-013-0542-8}
  {\path{doi:10.1007/s00211-013-0542-8}}.

\bibitem{Bao:2024:OEB}
W.~Bao, Y.~Ma, and C.~Wang.
\newblock Optimal error bounds on time-splitting methods for the nonlinear
  {S}chrödinger equation with low regularity potential and nonlinearity.
\newblock {\em Math. Models Methods Appl. Sci.}, 34(5):803--844, 2024.
\newblock \href {https://doi.org/10.1142/S0218202524500155}
  {\path{doi:10.1142/S0218202524500155}}.

\bibitem{Bao:2005:FTL}
W.~Bao and J.~Shen.
\newblock A fourth-order time-splitting {L}aguerre-{H}ermite pseudospectral
  method for {B}ose-{E}instein condensates.
\newblock {\em SIAM J. Sci. Comput.}, 26(6):2010--2028, 2005.
\newblock \href {https://doi.org/10.1137/030601211}
  {\path{doi:10.1137/030601211}}.

\bibitem{Bao:2008:AGP}
W.~Bao and J.~Shen.
\newblock A generalized-{L}aguerre-{H}ermite pseudospectral method for
  computing symmetric and central vortex states in {B}ose-{E}instein
  condensates.
\newblock {\em SIAM J. Sci. Comput.}, 227:9778--9793, 2008.
\newblock \href {https://doi.org/10.1016/j.jcp.2008.07.017}
  {\path{doi:10.1016/j.jcp.2008.07.017}}.

\bibitem{codesMatlab26}
Joackim Bernier, Ramona H\"aberli, and Gilles Vilmart.
\newblock Arbitrary high order splitting methods for linear schrödinger
  equations with non-trivial compatibility conditions.
\newblock {\em Yareta [Matlab source code]}, 2026.
\newblock \href {https://doi.org/10.26037/yareta:v72k2xeoh5bnzdqtyjbggigq2q}
  {\path{doi:10.26037/yareta:v72k2xeoh5bnzdqtyjbggigq2q}}.

\bibitem{Bertoli:2021:SOT}
G.~Bertoli, Ch. Besse, and G.~Vilmart.
\newblock Superconvergence of the {S}trang splitting when using the
  {C}rank-{N}icolson scheme for parabolic {PDE}s with {D}irichlet and oblique
  boundary conditions.
\newblock {\em Math. Comput.}, 90(332), 2021.
\newblock \href {https://doi.org/10.1090/mcom/3664}
  {\path{doi:10.1090/mcom/3664}}.

\bibitem{Bertoli:2020:SSM}
G.~Bertoli and G.~Vilmart.
\newblock Strang splitting method for semilinear parabolic problems with
  inhomogeneous boundary conditions: a correction based on the flow of the
  nonlinearity.
\newblock {\em SIAM J. Sci. Comput.}, 42(3):A1913--A1934, 2020.
\newblock \href {https://doi.org/10.1137/19M1257081}
  {\path{doi:10.1137/19M1257081}}.

\bibitem{Blanes:1999:SIW}
S.~Blanes, F.~Casas, and J.~Ros.
\newblock Symplectic integration with processing: a general study.
\newblock {\em SIAM J. Sci. Comput.}, 21(2):711--727, 1999.
\newblock \href {https://doi.org/10.1137/S1064827598332497}
  {\path{doi:10.1137/S1064827598332497}}.

\bibitem{Blanes:2000:SMF}
S.~Blanes and P.C. Moan.
\newblock Splitting methods for the time-dependent {S}chrödinger equation.
\newblock {\em Phys. Letters A}, 265(1--2):35--42, 2000.
\newblock \href {https://doi.org/10.1016/S0375-9601(99)00866-X}
  {\path{doi:10.1016/S0375-9601(99)00866-X}}.

\bibitem{Bona:2018:NBV}
J.~L. Bona, S.-M. Sun, and B.-Y. Zhang.
\newblock Nonhomogeneous boundary-value problems for one-dimensional nonlinear
  {S}chrödinger equations.
\newblock {\em J. Math. Pures Appl.}, 109:1--66, 2018.
\newblock \href {https://doi.org/10.1016/j.matpur.2017.11.001}
  {\path{doi:10.1016/j.matpur.2017.11.001}}.

\bibitem{Butcher:1969:TEO}
J.~C. Butcher.
\newblock {\em The effective order of {R}unge-{K}utta methods}, volume 109 of
  {\em Lecture Notes in Math.}
\newblock In Morris, J. L. (eds) \emph{Proceedings of conference on the
  numerical solutions of differential equations}, pages 133--139, 1969.

\bibitem{Einkemmer:2015:OOR}
L.~Einkemmer and A.~Ostermann.
\newblock Overcoming order reduction in diffusion-reaction splitting. {P}art 1:
  {D}irichlet boundary conditions.
\newblock {\em SIAM J. Sci. Comput.}, 37(3):A1577--A1592, 2015.
\newblock \href {https://doi.org/10.1137/140994204}
  {\path{doi:10.1137/140994204}}.

\bibitem{Einkemmer:2016:OOR}
L.~Einkemmer and A.~Ostermann.
\newblock Overcoming order reduction in diffusion-reaction splitting. {P}art 2:
  Oblique boundary conditions.
\newblock {\em SIAM J. Sci. Comput.}, 38(6):A3741--A3757, 2016.
\newblock \href {https://doi.org/10.1137/16M1056250}
  {\path{doi:10.1137/16M1056250}}.

\bibitem{Faou:2012:GNI}
E.~Faou.
\newblock {\em Geometric numerical integration and {S}chr\"odinger equations}.
\newblock Zurich Lectures in Advanced Mathematics. European Mathematical
  Society (EMS), Z\"urich, 2012.
\newblock \href {https://doi.org/10.4171/100} {\path{doi:10.4171/100}}.

\bibitem{Hairer:2010:GNI}
E.~Hairer, Ch. Lubich, and G.~Wanner.
\newblock {\em Geometric numerical integration. Structure-preserving algorithms
  for ordinary differential equations}, volume~31 of {\em Springer Series in
  Comp. Math.}
\newblock Springer, Heidelberg, 2006.
\newblock Second edition.

\bibitem{Haeberli:2026:OTO}
R.~Häberli.
\newblock Overcoming the order barrier two in splitting methods when applied to
  semilinear parabolic problems with non-periodic boundary conditions.
\newblock {\em SMAI J. Comput. Math.}, 12:269--288, 2026.
\newblock \href {https://doi.org/10.5802/smai-jcm.149}
  {\path{doi:10.5802/smai-jcm.149}}.

\bibitem{Jahnke:2000:EBF}
T.~Jahnke and Ch. Lubich.
\newblock Error bounds for exponential operator splittings.
\newblock {\em BIT Numer. Math.}, 40:735--744, 2000.
\newblock \href {https://doi.org/10.1023/A:1022396519656}
  {\path{doi:10.1023/A:1022396519656}}.

\bibitem{Ji:2024:LRF}
L.~Ji, A.~Ostermann, F.~Rousset, and K.~Schratz.
\newblock Low regularity full error estimates for the cubic nonlinear
  {S}chrödinger equation.
\newblock {\em SIAM J. Numer. Anal.}, 62(5):2071--2086, 2024.
\newblock \href {https://doi.org/10.1137/23M1619617}
  {\path{doi:10.1137/23M1619617}}.

\bibitem{Ji:2025:LRE}
L.~Ji, A.~Ostermann, F.~Rousset, and K.~Schratz.
\newblock Low regularity error estimates for the time integration of 2{D}
  {NLS}.
\newblock {\em IMA J. Numer. Anal. (to appear)}, 2025.
\newblock \href {https://doi.org/10.48550/arXiv.2301.10639}
  {\path{doi:10.48550/arXiv.2301.10639}}.

\bibitem{Kato:1985:ADE}
T.~Kato.
\newblock {\em Abstract differential equations and nonlinear mixed problems}.
\newblock Lezioni fermiane, 1985, Pisa.

\bibitem{Kawashima:1992:GEA}
S.~Kawashima and Y.~Shibata.
\newblock Global existence and exponential stability of small solutions to
  nonlinear viscoelasticity.
\newblock {\em Commun. Math. Phys.}, 148:189--208, 1992.
\newblock \href {https://doi.org/10.1007/BF02102372}
  {\path{doi:10.1007/BF02102372}}.

\bibitem{Lasiecka:1986:NHB}
I.~Lasiecka, J.-L Lions, and R.~Triggiani.
\newblock Non homogeneous boundary value problems for second order hyperbolic
  operators.
\newblock {\em J. Math. Pures Appl.}, 65:149--192, 1986.
\newblock URL: \url{https://api.semanticscholar.org/CorpusID:118945056}.

\bibitem{Leimkuhler:2004:SHD}
B.~Leimkuhler and S.~Reich.
\newblock {\em Simulating {H}amiltonian dynamics}.
\newblock Cambridge Monographs on Appl. and Comput. Math. 14. Cambridge
  University Press, Cambridge, 2004.

\bibitem{Lubich:2008:FQT}
Ch. Lubich.
\newblock {\em From quantum to classical molecular dynamics: reduced models and
  numerical analysis}.
\newblock Zürich Lectures in Adv. Math. European Mathematical Society (EMS),
  Zürich, 2008.

\bibitem{Neri:1985:LAA}
F.~Neri.
\newblock {L}ie algebras and canonical integration.
\newblock University of Maryland, draft, 1985.

\bibitem{Perrin:2025:COR}
T.~Perrin.
\newblock Change of regularity in controllability and observability of systems
  of wave equations.
\newblock {\em ESAIM: COCV}, 31:1--58, 2025.
\newblock \href {https://doi.org/10.1051/cocv/2025033}
  {\path{doi:10.1051/cocv/2025033}}.

\bibitem{Sulem:1999:TNS}
C.~Sulem and P.-L. Sulem.
\newblock {\em The nonlinear {S}chrödinger equation: self focusing and wave
  collapse}, volume 139 of {\em Appl. Math. Sci.}
\newblock Springer-Verlag New York, 1999.

\bibitem{Thalhammer:2008:HOE}
M.~Thalhammer.
\newblock High-order exponential operator splitting methods for time-dependent
  {S}chrödinger equations.
\newblock {\em IMA J. Numer. Anal.}, 46(4):2022--2038, 2008.
\newblock \href {https://doi.org/10.1137/060674636}
  {\path{doi:10.1137/060674636}}.

\bibitem{Wells:2019:AFA}
D.~Wells and H.~Quiney.
\newblock A fast and adaptable method for high accuracy integration of the
  time-dependent {S}chrödinger equation.
\newblock {\em Sci- Reports}, 782(9), 2019.
\newblock \href {https://doi.org/10.1038/s41598-018-37382-0}
  {\path{doi:10.1038/s41598-018-37382-0}}.

\bibitem{Yoshida:1990:COH}
H.~Yoshida.
\newblock Construction of higher order symplectic integrators.
\newblock {\em Phys. Letters A}, 150(5-7):262--268, 1990.
\newblock \href {https://doi.org/10.1016/0375-9601(90)90092-3}
  {\path{doi:10.1016/0375-9601(90)90092-3}}.

\end{thebibliography}

\end{document}